\documentclass[12pt]{amsart}
\usepackage{amscd}
\usepackage{amsfonts}
\usepackage{graphicx}
\usepackage{amsmath}
\usepackage{amssymb}
\usepackage{verbatim} 
\usepackage[margin=1.2 in]{geometry}
\providecommand{\U}[1]{\protect\rule{.1in}{.1in}}
\providecommand{\U}[1]{\protect\rule{.1in}{.1in}}
\newtheorem{theorem}{Theorem}

\newtheorem{condition}[theorem]{Condition}

\newtheorem{definition}[theorem]{Definition}
\newtheorem{example}[theorem]{Example}

\newtheorem{lemma}[theorem]{Lemma}

\newtheorem{proposition}[theorem]{Proposition}
\newtheorem{remark}[theorem]{Remark}

\begin{document}
\title[Band-limited Vectors on Nilpotent Lie Groups]{Sinc-Type Functions on a Class of Nilpotent Lie Groups}
\author[V. Oussa]{Vignon Oussa}
\address{Dept.\ of Mathematics \& Computer Science\\
Bridgewater State University\\
Bridgewater, MA 02325 U.S.A.\\}
\email{vignon.oussa@bridgew.edu}
\date{May 2013}

\begin{abstract}
Let $N$ be a simply connected, connected
nilpotent Lie group with the following assumptions. Its Lie algebra $\mathfrak{n}$ is an $n$-dimensional vector space over the reals. Moreover,  $\mathfrak{n=z}\oplus\mathfrak{b}\oplus\mathfrak{a}$, $\mathfrak{z}$
is the center of $\mathfrak{n}$,
$\mathfrak{z}   =\mathbb{R}Z_{n-2d}\oplus\mathbb{R}Z_{n-2d-1}\oplus\cdots\oplus
\mathbb{R}Z_{1}, \mathfrak{b}   =\mathbb{R}Y_{d}\oplus\mathbb{R}
Y_{d-1}\oplus\cdots\oplus\mathbb{R}Y_{1},
\mathfrak{a}    =\mathbb{R}X_{d}\oplus\mathbb{R}X_{d-1}\oplus\cdots\oplus \mathbb{R}X_{1}.$ Next, assume $\mathfrak{z}\oplus\mathfrak{b}$ is a maximal commutative ideal of
$\mathfrak{n},$ $\left[  \mathfrak{a,b}\right]  \subseteq\mathfrak{z},$  and  $\mathrm{det}\left([X_i,Y_j]\right)_{1\leq i,j\leq d}$  is a non-trivial
homogeneous polynomial defined over the ideal $\left[  \mathfrak{n,n}\right]
\subseteq\mathfrak{z}.$  We do not assume that $[\mathfrak{a},\mathfrak{a}]$ is generally trivial. We obtain some precise description of band-limited spaces which are sampling subspaces of $L^2(N)$ with respect to some discrete set $\Gamma$.  The set $\Gamma$ is explicitly constructed by fixing a strong Malcev basis for $\mathfrak{n}.$ We provide sufficient conditions for which a function $f$ is determined from its sampled values on $(f(\gamma))_{\gamma \in\Gamma}.$ We also provide an explicit formula for the corresponding sinc-type functions. Several examples are also computed in the paper.
\end{abstract}
\maketitle
\section{Introduction}
Let $\Omega$ be a positive number. A function $f$ in $L^{2}\left(\mathbb{R}\right)  $ is called $\Omega$-band-limited if its Fourier transform:
$\mathcal{F}f\left(  \lambda\right)$ is equal to zero for almost every  $\lambda$ outside of the interval $[-\Omega,\Omega].$ According to the well-known \textbf{Shannon-Whittaker-Kotel'nikov} theorem, $f$ is determined
by its sampled values $\left(f\left(  \frac{\pi n}{\Omega}\right)\right)_{n\in \mathbb{Z}} .$ In fact, for any function in the Hilbert space $$
\mathbf{H}_{\Omega}=\left\{  f\in L^{2}\left(\mathbb{R}\right)  :\text{ support }\left(  \mathcal{F}f\right)  \subset\left[
-\Omega,\Omega\right]  \right\}$$we
have the following reconstruction formula
\begin{equation}
f\left(  x\right)  =\sum_{n\in\mathbb{Z}}f\left(  \frac{\pi n}{\Omega}\right)  \frac{\sin\left(  \pi\left(
x-k\right)  \right)  }{\left(  \pi\left(  x-k\right)  \right)  }
\label{construction}
\end{equation}
and we say that $\mathbf{H}_{\Omega}$ is a \textbf{sampling subspace} of $L^{2}\left(\mathbb{R}\right)  $ with respect to the lattice $\frac{\pi}{\Omega}\mathbb{Z}.$ A relatively novel problem in abstract harmonic analysis has been to find analogues of (\ref{construction}) for other locally compact groups \cite{osc,Fuhr cont,pese,pesenson,oussa,oussa1}. Any attempt to generalize the given formula above leads to several obstructions. 
\begin{enumerate}
\item Let us recall that a unitary representation $\pi$ of a locally compact group $G$ is a factor representation if the center of the commutant algebra of $\pi$ is trivial, in the sense that it consists of scalar multiples of the identity operator. Moreover, $G$ is said to be a type I group if every factor representation of the group is a direct sum of copies of some irreducible representation. In general, harmonic analysis on non-type I groups is not well understood (see \cite{Folland}). For example the classification (in a reasonable sense) of the unitary dual of a non-type I group is a hopeless quest. Thus, for non-type I groups, it is not clear how to define a natural notion of band-limitation.
\item For type I groups, there exist a group Fourier transform, a Plancherel theory, and a natural notion of band-limitation. In general, the Plancherel transform intertwines the left regular representation of the group with a direct integral of unitary irreducible representations occurring with some multiplicities.  Although there is a nice notion of band-limitation available, the fact that we have to deal with multiplicities (unlike the abelian case) is a serious obstruction that needs to be addressed. 
\item Even if we manage to deal with the issues related to the presence of multiplicity functions of the irreducible representations occurring in the decomposition of the left regular representation of the group, the irreducible representations which are occurring are not just characters like in the abelian case. We often have to deal with some Hilbert Schimdt operators whose actions are usually not well understood.
\item It is not clear what the sampling sets should be in general. Actually, in many examples it turns out that requiring the sampling sets to be groups or lattices is too restrictive. 
\end{enumerate}

Let $G$ be a locally compact group, and let
$\Gamma$ be a discrete subset of $G.$ Let $\mathbf{H}$ be a left-invariant
closed subspace of $L^{2}\left(  G\right)  $ consisting of continuous
functions. According to Definition $2.51,$ \cite{Fuhr cont},  a Hilbert space $\mathbf{H}$ is a \textbf{sampling space} with respect to
$\Gamma$ if the following properties hold. First, the mapping
\[
R_{\Gamma}:\mathbf{H}\xrightarrow{\hspace*{1cm}} l^{2}\left(  \Gamma\right)  ,\text{
}R_{\Gamma}f=\left(  f\left(  \gamma\right)  \right)  _{\gamma\in\Gamma}%
\]
is an isometry (or a scalar multiple of an isometry). In other words, for all $f\in\mathbf{H},$ $
\sum_{\gamma\in\Gamma}\left\vert f\left(  \gamma\right)  \right\vert
^{2}=\left\Vert f\right\Vert _{\mathbf{H}}^{2}.$ Secondly, there exists a vector $s\in\mathbf{H}$ such that for any vector
$f\in\mathbf{H},$ we have the following expansion $
f\left(  x\right)  =\sum_{\gamma\in\Gamma}f\left(  \gamma\right)  s\left(
\gamma^{-1}x\right)$ with convergence in the $L^{2}$-norm of $\mathbf{H.}$ The function $s$ is
called a \textbf{sinc-type} function. We remark that there are several versions of definitions of sampling spaces. The definition which is usually encountered in the literature only requires the restriction map $R_{\Gamma}$ to be a bounded map with a bounded inverse. 

Let $\left(  \pi,\mathbf{H}_{\pi
}\right)  $ denote a strongly continuous unitary representation of a locally
compact group $G.$ We say that the representation $\left(  \pi,\mathbf{H}
_{\pi}\right)  $ is \textbf{admissible} if and only if the map
\[
V_{\phi}:\mathbf{H_{\pi}\xrightarrow{\hspace*{1cm}}}L^{2}\left(  G\right)  ,\text{ }V_{\phi}\psi\left(  x\right)  =\left\langle \psi,\pi\left(  x\right)  \phi
\right\rangle
\]
defines an isometry of $\mathbf{H}_{\pi}$ into $L^{2}\left(  G\right)  ,$ and we say
that $\phi$ is an \textbf{admissible vector} or a \textbf{continuous wavelet}. It is known that if $\pi$ is the left regular representation of $G,$ and if
$G$ is connected and type I, then $\pi$ is admissible if and only if $G$ is
nonunimodular (See \cite{Fuhr cont} Theorem 4.23). The following fact is proved in Proposition 2.54 in \cite{Fuhr
cont}.  Let $\phi$ be an admissible vector for $\left(  \pi,\mathbf{H}_{\pi
}\right)  $ such that $\pi\left(  \Gamma\right)  \phi$ is a Parseval frame$.$
Then the Hilbert space $V_{\phi}\left(  \mathbf{H}_{\pi}\right)  $
is a sampling space, and $V_{\phi}\left(  \phi\right)  $ is the associated
\textbf{sinc-type} function for $V_{\phi}\left(  \mathbf{H}_{\pi}\right) .$\vskip 0.5 cm

Since non-commutative nilpotent Lie groups are very close to commutative Lie groups in their group structures, then it seems reasonable to conjecture that (\ref{construction}) extends to a large class of simply connected, connected non-commutative nilpotent Lie groups, and that this class of groups admits sampling susbpaces which resemble $\mathbf{H}_{\Omega}.$  Presently, we do not have a complete characterization of this class of nilpotent Lie groups.  However, we have some partial answers which allow us say that this class of nilpotent Lie groups is larger than the class considered in \cite {oussa} and \cite{oussa1}. The main purpose of this paper is to present the proof of this new result. 

We remark that reconstruction
theorems for the Heisenberg group (the simplest example of a connected, simply
connected non-commutative nilpotent Lie group) were obtained by F\"{u}hr
\cite{Fuhr cont}, and Currey and Mayeli in \cite{Currey}.  Other relevant sources are \cite{Mayeli, FuhrMayeli, pesenson}. In his work, F\"{u}hr
developed a natural concept of band-limitation on the space of square-integrable functions over the Heisenberg group.  Using the fact that the Plancherel measure of the Heisenberg group is supported on $\mathbb{R}^{\ast},$ he defined a band-limited Hilbert space over the Heisenberg group to be a space of square-integrable functions whose Plancherel transforms are supported on a fixed bounded subset of $\mathbb{R}^{\ast}.$ He was then able to provide characteristics of sampling spaces with respect to some integer lattices of the Heisenberg group. Furthermore, he provides an explicit Sinc type function in Theorem $6.18$ \cite{Fuhr cont}. For a larger class of step-two nilpotent Lie groups of the type $\mathbb{R}^{n-d}\rtimes \mathbb{R}^{d}$ which is properly contained in the class
of groups considered in this paper, we also obtained some sampling theorems in
\cite{oussa} and \cite{oussa1}.  Since we are dealing with some non-commutative groups, it is worth noticing the following. First, unlike the commutative case, the
corresponding Fourier transforms of the groups are operator-valued transforms.
Secondly, the left regular representations of the groups decompose into direct
integrals of infinite dimensional irreducible representations, each occurring
with infinite multiplicities. We will only be concerned with the multiplicity-free case in this paper. 
\subsection{Overview of the Paper}
Let $N$ be a simply connected, connected nilpotent Lie group. Let
$\mathfrak{n}$ be its Lie algebra satisfying the following.
\begin{condition}
\label{condition}
\end{condition}
\begin{description}
\item[1] $\mathfrak{n=z}\oplus\mathfrak{b}\oplus\mathfrak{a}$, $\mathfrak{z}$
is the center of $\mathfrak{n,}$%
\begin{align*}
\mathfrak{z}  &  =%
\mathbb{R}
Z_{n-2d}\oplus%
\mathbb{R}
Z_{n-2d-1}\oplus\cdots\oplus%
\mathbb{R}
Z_{1},\\
\mathfrak{b}  &  =%
\mathbb{R}
Y_{d}\oplus%
\mathbb{R}
Y_{d-1}\oplus\cdots\oplus%
\mathbb{R}
Y_{1},\\
\mathfrak{a}  &  =%
\mathbb{R}
X_{d}\oplus%
\mathbb{R}
X_{d-1}\oplus\cdots\oplus%
\mathbb{R}
X_{1}%
\end{align*}

\item[2] $\mathfrak{z}\oplus\mathfrak{b}$ is a maximal commutative ideal of
$\mathfrak{n}$

\item[3] $\left[  \mathfrak{a,b}\right]  \subseteq\mathfrak{z}$

\item[4] Given the square matrix of order $d$
\begin{equation}
S=\left[
\begin{array}
[c]{ccc}%
\left[  X_{1},Y_{1}\right]  & \cdots & \left[  X_{d},Y_{1}\right] \\
\vdots & \ddots & \vdots\\
\left[  X_{d},Y_{1}\right]  & \cdots & \left[  X_{d},Y_{d}\right]
\end{array}
\right]  \label{S}%
\end{equation}
the homogeneous polynomial $\det\left(  S\right)  $ is a non-trivial
polynomial defined over the ideal $\left[  \mathfrak{n,n}\right]
\subseteq\mathfrak{z}.$
\end{description}

Define the discrete set
\begin{equation}
\Gamma=\exp\left(  \mathbb{Z}Z_{n-2d}\right)  \cdots\exp\left(  \mathbb{Z}%
Z_{1}\right)  \exp\left(  \mathbb{Z}Y_{d}\right)  \cdots\exp\left(
\mathbb{Z}Y_{1}\right)  \exp\left(  \mathbb{Z}X_{d}\right)  \cdots\exp\left(
\mathbb{Z}X_{1}\right) \label{Gamma}%
\end{equation}
which is a subset of $N$ and define a matrix-valued function on $\mathfrak{z}^{\ast}$ as follows
\[
S(\lambda)=\left[
\begin{array}
[c]{ccc}%
\lambda\left[  X_{1},Y_{1}\right]   & \cdots & \lambda\left[  X_{1}%
,Y_{d}\right]  \\
\vdots &  & \vdots\\
\lambda\left[  X_{d},Y_{1}\right]   & \cdots & \lambda\left[  X_{d}%
,Y_{d}\right]
\end{array}
\right]  .
\]
In order to have a reconstruction formula, we will need a very specific definition of band-limitation. Indeed, we prove that there exists a fundamental domain $\mathbf{K}$ for $\mathbb{Z}%
^{n-2d}\cap\mathfrak{z}^{\ast}$ such that $\mathbf{I}=\mathbf{F}\cap\mathbf{K}$ is a set
of positive measure in $\mathfrak{z}^{\ast}$ and 
\[
\mathbf{F}=\left\{  \lambda\in\mathfrak{z}^{\ast}:\left\vert \det S\left(
\lambda\right)  \right\vert \leq1,\det S\left(  \lambda\right)  \neq0,\text{
and }\left\Vert S\left(  \lambda\right)  ^{Tr}\right\Vert_{\infty} <1\right\}.
\]
Let $\{\mathbf{u}_{\lambda} : \lambda \in \mathbf{I}\}$ be a measurable field of unit vectors in $L^2(\mathbb{R}^d).$ Let $\mathbf{H_{\mathbf{u},\mathbf{I}}}$ be a left-invariant subspace of $L^{2}\left(  N\right)  $
such that
\[
\mathbf{H_{\mathbf{u},\mathbf{I}}}=\left\{
\begin{array}
[c]{c}%
f\in L^{2}\left(  N\right)  :\widehat{f}(\lambda)=v_{\lambda}\otimes
\mathbf{u}_{\lambda}\text{ is a rank-one operator }\\
\text{in $L^{2}(\mathbb{R}^{d})\otimes L^{2}(\mathbb{R}^{d})$ and
\textrm{support of} }\widehat{f}\subseteq\mathbf{I}%
\end{array}
\right\}  .
\]
We have two main results. If $1_{d,d}$ stands for the identity matrix of order $d$, and $0_{d,d}$ is the zero matrix of order $d,$ then

\begin{theorem} \label{M1}The unitary dual of $N$ is up to a null set exhausted by the set of irreducible
representations $
\left\{  \pi_{\lambda}:\lambda\in\mathfrak{z}^{\ast}\text{ and }\det S(\lambda)\neq 0 \right\}$ where, for $f\in L^{2}\left(
\mathbb{R}
^{d}\right)  $, $\pi_{\lambda}\left(  \Gamma_{1}  \right) f$ is a Gabor system of the type $
\mathcal{G}\left(  f,B\left(  \lambda\right)  \mathbb{Z}^{2d}\right)
=\left\{  e^{2\pi i\left\langle k,x\right\rangle }f\left(  x-n\right)
:\left(  n,k\right)  \in B\left(  \lambda\right)  \mathbb{Z}^{2d}\right\},$ where

$$
\Gamma_1=\exp\left(  \mathbb{Z}Y_{d}\right)  \cdots\exp\left(
\mathbb{Z}Y_{1}\right)  \exp\left(  \mathbb{Z}X_{d}\right)  \cdots\exp\left(
\mathbb{Z}X_{1}\right),
$$

\begin{equation}
\label{0}B\left(  \lambda\right)  =\left[
\begin{array}
[c]{cc}%
1_{d,d} & 0_{d,d}\\
-X\left(  \lambda\right)  & -S\left(  \lambda\right)
\end{array}
\right]  ,\text{ }S\left(  \lambda\right)  =\left(  \lambda\left[  X_{i}%
,Y_{i}\right]  \right)  _{1\leq i,j\leq d},
\end{equation}
and $X\left(  \lambda\right)  $ is a strictly upper triangular matrix with
entries in the dual of the vector space $\left[  \mathfrak{a},\mathfrak{a}%
\right]  ,$ with $X\left(  \lambda\right)  _{i,j}=\lambda\left[  X_{i}%
,X_{j}\right]  $ for $i<j.$ \end{theorem}

\begin{theorem} \label{M} There exists a function $f\in \mathbf{H}_{\mathbf{u},\mathbf{I}}$ such that the Hilbert subspace $V_f(\mathbf{H}_{\mathbf{u},\mathbf{I}})$ of $L^2(N)$ is a sampling space with sinc-type function $V_f(f)$.\end{theorem}
The paper is organized around our two main results. Theorem \ref{M1} is proved in the second section and Theorem \ref{M} is proved in the last section of the paper. Also, several interesting examples are given throughout the paper, to help the reader follow the stream of ideas presented. 
\section{The Unitary Dual of $N$ and the Proof of Theorem \ref{M1}}
Let us start by setting up some notation. In this paper, all representations
are strongly continuous and unitary. All sets are measurable. The characteristic
function of a set $E$ is written as $\chi_{E},$ and $L$ stands for the left
regular representation of a given locally compact group. If
$\mathfrak{a,b}$ are vector subspaces of some Lie algebra $\mathfrak{g}$, we
denote by $\left[  \mathfrak{a,b}\right]  $ the set of linear combinations of
the form $\left[  X,Y\right]  $ where $X\in\mathfrak{a}$ and $Y\in
\mathfrak{b}$. The linear dual of a finite-dimensional vector space $V$ is
denoted by $V^{\ast},$ and given two equivalent representations $\pi_{1}$ and $\pi_{2},$ we write $\pi_{1}\cong\pi_{2}.$ Given a matrix $M,$ the transpose of $M$ is written as $M^{Tr}.$
\subsection{The Unitary Dual of $N$}
Let $\mathfrak{n}$ be a nilpotent Lie algebra of dimension $n$ over $\mathbb{R}$
with corresponding Lie group $N=\exp\mathfrak{n}$. We assume that $N$ is
simply connected and connected. Let $\mathfrak{s}$ be a subset in $\mathfrak{n}$ and let $\lambda$ be a linear functional in $\mathfrak{n}^{\ast}$. We define the
corresponding sets $\mathfrak{s}^{\lambda}$ and $\mathfrak{s}\left(
\lambda\right)  $ such that $$
\mathfrak{s}^{\lambda}=\left\{  Z\in\mathfrak{n:}\text{ }\lambda\left[
Z,X\right]  =0\text{ for every }X\in\mathfrak{s}\right\}$$ and $\mathfrak{s}\left(  \lambda\right)  =\mathfrak{s}^{\lambda}%
\cap\mathfrak{s}.$ The ideal $\mathfrak{z}$ denotes the center of the Lie
algebra of $\mathfrak{n}$ and the coadjoint action on the dual of
$\mathfrak{n}$ is simply the dual of the adjoint action of $N$ on
$\mathfrak{n}$. Given $X\in\mathfrak{n}, \lambda \in \mathfrak{n}^{\ast}$, the coadjoint action is defined
multiplicatively as follows: $\exp X\cdot\lambda\left(  Y\right)
=\lambda\left(  Ad_{\exp-X}Y\right)  $. The following discussion describes
some stratification procedure of the dual of the Lie algebra $\mathfrak{n}$
which will be used to develop a precise Plancherel theory for $N.$ This theory
is also well exposed in \cite{Corwin}.
Let $\mathcal{B}=\left\{X_1,\cdots,X_n\right\}$ be a basis for $\mathfrak{n}.$ Let $$ \mathfrak{n}_1\subseteq  \mathfrak{n}_2\subseteq \cdots \subseteq   \mathfrak{n}_{n-1}\subseteq  \mathfrak{n}$$ be a sequence of subalgebras of $\mathfrak{n}.$ We recall that $\mathcal{B}$ is called a \textbf{strong Malcev basis} through $\mathfrak{n}_1, \mathfrak{n}_2, \cdots,   \mathfrak{n}_{n-1}, \mathfrak{n}$ if and only if the following holds.
\begin{enumerate}
\item The real span of $\left\{X_1,\cdots,X_k\right\}$ is equal to $\mathfrak{n}_k.$
\item Each $\mathfrak{n}_k$ is an ideal of $\mathfrak{n}.$
\end{enumerate}

We start by fixing a strong Malcev basis $\left\{  Z_{i}\right\}
_{i=1}^{n}$ for $\mathfrak{n}$ and we define an increasing sequence of ideals:
$\mathfrak{n}_{k}=\mathbb{R}$-$\mathrm{span}\left\{  Z_{i}\right\}  _{i=1}%
^{k}.$ Given any linear functional $\lambda\in\mathfrak{n}^{\ast},$ we
construct the following skew-symmetric matrix:
\begin{equation}
M\left(  \lambda\right)  =\left[  \lambda\left[  Z_{i},Z_{j}\right]  \right]
_{1\leq i,j,n}. \label{Ml}%
\end{equation}
It is easy to see that $\mathfrak{n}\left(  \lambda\right)
=\mathrm{nullspace}\left(  M\left(  \lambda\right)  \right)  .$ It is also
well-known that all coadjoint
orbits have a natural  symplectic smooth structure, and therefore are even-dimensional manifolds. Also, for each $\lambda\in$ $\mathfrak{n}%
^{\ast}$ there is a corresponding set $\mathbf{e}\left(  \lambda\right)
\subset\left\{  1,2,\cdots,n\right\}  $ of \textbf{jump indices} defined by
\[
\mathbf{e}\left(  \lambda\right)  =\left\{  1\leq j\leq n:\mathfrak{n}%
_{k}\text{ }\nsubseteq\text{ }\mathfrak{n}_{k-1}+\mathfrak{n}\left(
\lambda\right)  \right\}  .
\]
Naively speaking, the set $\mathbf{e}\left(  \lambda\right)  $ collects all
basis elements
\[
\left\{  B_{1},\cdots,B_{2d}\right\}  \subset\left\{  Z_{1},Z_{2}%
,\cdots,Z_{n-1},Z_{n}\right\}
\]
in the Lie algebra $\mathfrak{n}$ such that
$
\exp\left(
\mathbb{R}
B_{1}\right)  \cdots\exp\left(
\mathbb{R}
B_{2d}\right)  \cdot\lambda=G\cdot\lambda.
$
For each subset $\mathbf{e}$ $\subseteq\left\{  1,2,\cdots,n\right\}  ,$ the
set $\Omega_{\mathbf{e}}=\left\{  \lambda\in\mathfrak{n}^{\ast}:\mathbf{e}%
\left(  \lambda\right)  =\mathbf{e}\right\}  $ is algebraic and $N$-invariant.
Moreover, letting $
\xi=\left\{  \mathbf{e}\subseteq\left\{  1,2,\cdots,n\right\}  :\Omega
_{\mathbf{e}}\neq\emptyset\right\}$ then
\[
\mathfrak{n}^{\ast}=%
{\displaystyle\bigcup\limits_{\mathbf{e}\in\xi}}
\Omega_{\mathbf{e}}.
\]
The union of all \textbf{non-empty layers} defines a `stratification' of
$\mathfrak{n}^{\ast}$. It is known that there is a total ordering $\prec$ on the stratification for
which the minimal element is Zariski open and consists of orbits of maximal
dimension. Let $\mathbf{e}$ be a subset of $\left\{  1,2,\cdots,n\right\}  $
and define $M_{\mathbf{e}}\left(  \lambda\right)  =\left[  \lambda\left[
Z_{i},Z_{j}\right]  \right]  _{i,j\in\mathbf{e}}$ . The set $\Omega
_{\mathbf{e}}$ is also given as follows:
\begin{equation}
\Omega_{\mathbf{e}}=\left\{  \lambda\in\mathfrak{n}^{\ast}:\det M_{\mathbf{e}%
^{\prime}}\left(  \lambda\right)  =0\text{ for all }\mathbf{e}^{\prime}%
\prec\mathbf{e}\text{ and }\det M_{\mathbf{e}}\left(  \lambda\right)
\neq0\text{ }\right\}  . \label{omega}%
\end{equation}

Let us now fix an open and dense layer $\Omega=\Omega_{\mathbf{e}}%
\subset\mathfrak{n}^{\ast}.$ The following is standard. We define a
polarization subalgebra associated with the linear functional $\lambda$ by
$\mathfrak{p}(\lambda).$ $\mathfrak{p}(\lambda)$ is a maximal subalgebra
subordinated to $\lambda$ such that $\lambda\left(  \lbrack\mathfrak{p}%
(\lambda),\mathfrak{p}(\lambda)]\right)  =0$ and $\chi(\exp X)=e^{2\pi
i\lambda(X)}$ defines a character on $\exp(\mathfrak{p}(\lambda))$. It is
well-known that $\dim\left(  {\mathfrak{n}}(\lambda)\right)  =n-2d$ and
$\dim\left(  \mathfrak{n}/\mathfrak{p}(\lambda)\right)  =$ $d.$

According to the orbit method \cite{Corwin}, all irreducible representations
of $N$ are parametrized by the a set of coadjoint orbits. In order to describe
the unitary dual of $N$ and its Plancherel measure, we need to construct a
smooth cross-section $\Sigma$ which is homeomorphic to $\Omega/N$. Using
standard techniques described in \cite{Corwin}, we obtain $$
\Sigma=\left\{  \lambda\in\Omega:\lambda\left(  Z_{k}\right)  =0\text{ for all
}k\in\mathbf{e}\text{ }\right\} .$$ Defining for each linear functional $\lambda$ in the generic layer, a
character of $\exp\left(  \mathfrak{p}(\lambda)\right)  $ such that
$\chi_{\lambda}\left(  \exp X\right)  =e^{2\pi i\lambda\left(  X\right)  }$,
we realize almost all unitary irreducible representations of $N$ by induction
as follows.
\[
\pi_{\lambda}=\mathrm{Ind}_{\exp\left(  \mathfrak{p}(\lambda)\right)  }%
^{N}\left(  \chi_{\lambda}\right)
\]
and $\pi_{\lambda}$ acts in the Hilbert completion of the space
\begin{equation}
\mathbf{H}_{\lambda}=\left\{
\begin{array}
[c]{c}%
f:N\xrightarrow{\hspace*{1cm}}\mathbb{C} : f\left(  xy\right)  =\chi_{\lambda
}\left(  y\right)  ^{-1}f\left(  x\right)  \text{ for }y\in\exp\mathfrak{p}%
(\lambda)\mathfrak{,}\text{ }\\
\text{and }\int_{\frac{N}{\exp\left(  \mathfrak{p}(\lambda)\right)  }%
}\left\vert f\left(  x\right)  \right\vert ^{2}d\overline{x}<\infty
\end{array}
\right\}  \label{Hilbert}%
\end{equation}
endowed with the following inner product: $$
\left\langle f,f^{\prime}\right\rangle =\int_{\frac{N}{\exp\left(
\mathfrak{p}(\lambda)\right)  }}f\left(  n\right)  \overline{f^{\prime}\left(
n\right)  }d\overline{n}.$$ In fact, there is an obvious identification between the completion of
$\mathbf{H}_{\lambda}$ and the Hilbert space $L^{2}\left(  \frac{N}%
{\exp\left(  \mathfrak{p}(\lambda)\right)  }\right)  .$ We will come back to
this later.
\newline
We will now focus on the class of nilpotent Lie groups that we are concerned
with in this paper. 

\begin{example}
Let $N$ be a nilpotent Lie group with Lie algebra spanned by $Z_{1}%
,Z_{2},Y_{1},Y_{2},X_{1},X_{2}$ such that
\begin{align*}
\left[  X_{1},X_{2}\right]    & =Z_{1},\left[  X_{1},Y_{1}\right]  =Z_{1}\\
\left[  X_{2},Y_{2}\right]    & =Z_{1},\left[  X_{1},Y_{2}\right]  =Z_{2}\\
\left[  X_{2},Y_{1}\right]    & =Z_{2}.
\end{align*}
Then $\det\left(  S\right)  =Z_{1}^{2}-Z_{2}^{2}$ and it is clear that $N$
belongs to the class of groups considered here. 
\end{example}

\begin{example}
Let $N$ be a nilpotent Lie group with Lie algebra spanned by the vectors
\[
\left\{  Z_{3},Z_{2},Z_{1},Y_{3},Y_{2},Y_{1},X_{3},X_{2},X_{1}\right\}
\]
and the following non-trivial Lie brackets
\begin{align*}
\left[  X_{1},X_{2}\right]   &  =Z_{1},\left[  X_{1},X_{3}\right]
=Z_{2},\left[  X_{2},X_{3}\right]  =Z_{3}\\
\left[  X_{1},Y_{1}\right]   &  =Z_{1},\left[  X_{1},Y_{2}\right]
=Z_{2}-Z_{3},\left[  X_{1},Y_{3}\right]  =Z_{1}+Z_{2}\\
\left[  X_{2},Y_{1}\right]   &  =Z_{2},\left[  X_{2},Y_{2}\right]
=Z_{1}-Z_{2},\left[  X_{2},Y_{3}\right]  =Z_{2}-Z_{3}\\
\left[  X_{3},Y_{1}\right]   &  =Z_{3},\left[  X_{3},Y_{2}\right]
=Z_{1}+Z_{2},\left[  X_{3},Y_{3}\right]  =Z_{3}.
\end{align*}
Then $\mathfrak{a}$ does not commute, $\mathfrak{z}\oplus\mathfrak{b}$ is a
maximal commutative ideal of $\mathfrak{n,}$ $\left[  \mathfrak{a,b}\right]
\subseteq\mathfrak{z,}$
\[
S=\left[
\begin{array}
[c]{ccc}%
Z_{1} & Z_{2}-Z_{3} & Z_{1}+Z_{2}\\
Z_{2} & Z_{1}-Z_{2} & Z_{2}-Z_{3}\\
Z_{3} & Z_{1}+Z_{2} & Z_{3}%
\end{array}
\right]\]  and $$\det\left(  S\right)  =Z_{1}^{2}Z_{3}+Z_{1}Z_{2}%
^{2}+Z_{2}^{3}+Z_{2}^{2}Z_{3}-Z_{2}Z_{3}^{2}+Z_{3}^{3}\neq0.$$

\end{example}

Let us define
\begin{align*}
B_{1}  &  =Z_{n-2d},B_{2}=Z_{n-2d-1}\cdots,B_{n-2d}=Z_{1},B_{n-2d+1}%
=Y_{d},\text{ }\\
B_{n-2d+2}  &  =Y_{d-1}\cdots,B_{n-d}=Y_{1},\text{ }B_{n-d+1}=X_{d}%
,B_{n-d+2}=X_{d-1},\cdots,\text{ and }B_{n}=X_{1}.
\end{align*}

\begin{lemma}
Let $\lambda\in\mathfrak{z}^{\ast}.$ If $\det\left(  S\right)  $ is a
non-vanishing polynomial then $\mathfrak{n}\left(  \lambda\right)
=\mathfrak{z}$ for a.e. $\lambda\in\mathfrak{z}^{\ast}.$
\end{lemma}

\begin{proof}
First, let
\[
S\left(  \lambda\right)  =\left[
\begin{array}
[c]{ccc}%
\lambda\left[  X_{1},Y_{1}\right]  & \cdots & \lambda\left[  X_{1}%
,Y_{d}\right] \\
\vdots & \ddots & \vdots\\
\lambda\left[  X_{d},Y_{1}\right]  & \cdots & \lambda\left[  X_{d}%
,Y_{d}\right]
\end{array}
\right]  .
\]
We recall the definition of $S$ from (\ref{S}). Clearly for a.e. $\lambda
\in\mathfrak{z}^{\ast},$ $S$ and $S\left(  \lambda\right)  $ have the same
rank, which is equal to $d$ on a dense open subset of the linear dual of the
central ideal of the Lie algebra. In fact, we call $d$ the generic rank of the
matrix $S(\lambda).$ Also, we recall that $\mathfrak{n}\left(  \lambda\right)
$ is the null-space of $M\left(  \lambda\right)  $ which is defined in
(\ref{Ml}) as follows
\[
M\left(  \lambda\right)  =\left[  \lambda\left[  B_{i},B_{j}\right]  \right]
_{1\leq i,j\leq n}=\left[
\begin{array}
[c]{ccc}%
0_{n-2d,n-2d} & 0_{n-2d,d} & 0_{n-2d,d}\\
0_{d,n-2d} & 0_{d,d} & S^{\prime}\left(  \lambda\right) \\
0_{d,n-2d} & -S^{\prime}\left(  \lambda\right)  & R\left(  \lambda\right)
\end{array}
\right]  ;
\]
$0_{p,q}$ stands for the $p\times q$ zero matrix,
\[
S^{\prime}\left(  \lambda\right)  =\left[
\begin{array}
[c]{ccc}%
\lambda\left[  Y_{d},X_{d}\right]  & \cdots & \lambda\left[  Y_{d}%
,X_{1}\right] \\
\vdots & \ddots & \vdots\\
\lambda\left[  Y_{1},X_{d}\right]  & \cdots & \lambda\left[  Y_{1}%
,X_{1}\right]
\end{array}
\right]  ,\]
and 
\[R\left(  \lambda\right)  =\left[
\begin{array}
[c]{ccc}%
\lambda\left[  X_{d},X_{d}\right]  & \cdots & \lambda\left[  X_{d}%
,X_{1}\right] \\
\vdots & \ddots & \vdots\\
\lambda\left[  X_{1},X_{d}\right]  & \cdots & \lambda\left[  X_{1}%
,X_{1}\right]
\end{array}
\right]
\]
Since the first $n-2d$ columns of the matrix $M\left(  \lambda\right)  $ are
zero vectors, and since the remaining $2d$ columns are linearly independent
then the nullspace of $M\left(  \lambda\right)  $ is equal to the center of
the algebra $\mathfrak{n}$ which is a vector space spanned by $n-2d$ vectors.
\end{proof}

Fix $
\mathbf{e=}\left\{  n-2d+1,n-2d+2,\cdots,n\right\} .$ It is not too hard to see that the corresponding layer $
\Omega=\Omega_{\mathbf{e}}=\left\{  \lambda\in\mathfrak{n}^{\ast}:\det\left(
S\left(  \lambda\right)  \right)  \neq0\right\}$ is a Zariski open and dense set in $\mathfrak{n}^{\ast}.$ Next, the manifold \begin{equation}\label{Sigma}
\Sigma=\left\{  \lambda\in\Omega:\lambda\left(  \mathfrak{b\oplus a}\right)
=0\right\}\end{equation} gives us an almost complete parametrization of the unitary dual of $N$ since
it is a cross-section for the coadjoint orbits in the layer $\Omega$.
Moreover, we observe that $\Sigma$ is homeomorphic with a Zariski open subset
of $\mathfrak{z}^{\ast}.$ In order to obtain a realization of the irreducible
representation corresponding to each linear functional in $\Sigma,$ we will
need to construct a corresponding polarization subalgebra \cite{Corwin}.

The following lemma is in fact the first step toward a precise computation of
the unitary dual of $N.$ Put $
\mathfrak{p}=\mathfrak{z}\oplus\mathfrak{b}.$

\begin{lemma}
For every $\lambda\in\Sigma,$ a corresponding polarization subalgebra is given
by the ideal $\mathfrak{p}$
\end{lemma}

\begin{proof}
Since $\mathfrak{p}$ is a commutative algebra then clearly $\lambda\left[
\mathfrak{p,p}\right]  =\left\{  0\right\}  .$ In order to prove the lemma, it
suffices to show that $\mathfrak{p}$ is a maximal algebra such that
$\lambda\left[  \mathfrak{p,p}\right]  =\left\{  0\right\}  $. Let us suppose
by contradiction that it is not. There exists a non-zero vector $A\in\mathfrak{a}$ such that
$
\mathfrak{p}\subsetneq\mathfrak{p\oplus\mathbb{R}}A
$
and $\left[  \mathfrak{p\oplus%
\mathbb{R}
}A,\mathfrak{p\oplus%
\mathbb{R}
}A\right]  $ is a zero vector space. However, $
\left[  \mathfrak{p\oplus%
\mathbb{R}
}A,\mathfrak{p\oplus%
\mathbb{R}
}A\right]  =\left[  \mathfrak{b\oplus%
\mathbb{R}
}A,\mathfrak{b\oplus%
\mathbb{R}
}A\right]  =\left\{  0\right\}.$ So there exists an element of $\mathfrak{a}$ which commutes with all the vectors
$Y_{k},1\leq k\leq d.$ This contradicts the fourth assumption in Condition \ref{condition}.
\end{proof}

We recall the definition of the discrete set $\Gamma$ given in (\ref{Gamma})
and we define the discrete set
\[
\Gamma_{1}=\exp\mathbb{Z}Y_{d}\cdots\exp\mathbb{Z}Y_{1}\exp\mathbb{Z}%
X_{d}\cdots\exp\mathbb{Z}X_{1}\subset N.
\]
We observe that $\Gamma_{1}$ is not a group but is naturally identified with
the set $\mathbb{Z}^{2d}.$ 

Next, since $N$ is a non-commutative group, the following remark is in order.
Let $A$ be a set, and $\mathrm{Sym}\left(  A\right)  $ be the group of
permutation maps of $A.$

\begin{remark}
Let $\sigma\in\mathrm{Sym}\left(  \left\{  1,\cdots,d\right\}  \right)  $ be a
permutation map. Since $\mathfrak{a}$ is not commutative, it is clear that in
general
\[
\Gamma\neq\exp\left(  \mathbb{Z}Z_{n-2d}\right)  \cdots\exp\left(
\mathbb{Z}Z_{1}\right)  \exp\left(  \mathbb{Z}Y_{d}\right)  \cdots\exp\left(
\mathbb{Z}Y_{1}\right)  \exp\left(  \mathbb{Z}X_{\sigma\left(  d\right)
}\right)  \cdots\exp\left(  \mathbb{Z}X_{\sigma\left(  1\right)  }\right)
\]
However, for arbitrary permutation maps $\sigma_{1},\sigma_{2}$ such that
$\sigma_{1}\in\mathrm{Sym}\left(  \left\{  1,\cdots,n-2d\right\}  \right)  $
and $\sigma_{2}\in\mathrm{Sym}\left(  \left\{  1,\cdots,d\right\}  \right)  $,
the following holds true:
\[
\Gamma=\exp\left(  \mathbb{Z}Z_{\sigma_{1}\left(  n-2d\right)  }\right)
\cdots\exp\left(  \mathbb{Z}Z_{\sigma_{1}\left(  1\right)  }\right)
\exp\left(  \mathbb{Z}Y_{\sigma_{2}\left(  d\right)  }\right)  \cdots
\exp\left(  \mathbb{Z}Y_{\sigma_{2}\left(  1\right)  }\right)  \exp\left(
\mathbb{Z}X_{d}\right)  \cdots\exp\left(  \mathbb{Z}X_{1}\right)  .
\]

\end{remark}

Now, let $x=\left(  x_{1},x_{2},\cdots,x_{d}\right)  \in%
\mathbb{R}
^{d}.$

\begin{proposition}
\label{irreducible} The unitary dual of $N$ is given by
\[
\left\{  \pi_{\lambda}=\mathrm{Ind}_{\exp\left(  \mathfrak{z}\oplus
\mathfrak{b}\right)  }^{N}\left(  \chi_{\lambda}\right)  :\lambda=\left(
\lambda_{1},\lambda_{2},\cdots,\lambda_{n-2d},0,\cdots,0\right)  ,\text{ and
}\det\left(  S\left(  \lambda\right)  \right)  \neq0\right\}
\]
which we realize as acting in $L^{2}\left(
\mathbb{R}
^{d}\right)  $ as follows.%
\begin{align*}
&  \pi_{\lambda}\left(  \exp\left(  z_{n-2d}Z_{n-2d}\right)  \cdots\exp\left(
z_{1}Z_{1}\right)  \exp\left(  l_{d}Y_{d}\right)  \cdots\exp\left(  l_{1}%
Y_{1}\right)  \exp\left(  m_{d}X_{d}\right)  \cdots\exp\left(  m_{1}%
X_{1}\right)  \right)  \phi\left(  x\right) \\
&  =e^{ 2\pi i \sum_{j=1}^{d}s_{j}\lambda Z_{j} }
e^{  -2\pi i\sum_{j=1}^{d}\sum_{k=1}^{d}x_{k}l_{j}%
\lambda\left[  X_{k},Y_{j}\right]} e^{   -2\pi i\lambda\left(  \sum_{j=2}^{d}\sum_{r=1}^{j-1}m_{j}%
x_{r}\left[  X_{r},X_{j}\right]  \right)}  f\left(  x_{1}-m_{1}%
,\cdots,x_{d}-m_{d}\right)  .
\end{align*}

\end{proposition}

\begin{proof}
We recall that $\pi_{\lambda}$ acts in the Hilbert completion of
\begin{equation}
\mathbf{H}_{\lambda}=\left\{
\begin{array}
[c]{c}%
f:N\xrightarrow{\hspace*{1cm}}\mathbb{C}\text{ such that }f\left(  xy\right)  =\chi_{\lambda
}\left(  y\right)  ^{-1}f\left(  x\right)  \text{ for }y\in\exp\left(
\mathfrak{z}\oplus\mathfrak{b}\right) \\
\text{and }x\in N/\exp\left(  \mathfrak{z}\oplus\mathfrak{b}\right)  \text{
and }\int_{N/\exp\left(  \mathfrak{z}\oplus\mathfrak{b}\right)  }\left\vert
f\left(  x\right)  \right\vert ^{2}d\overline{x}<\infty
\end{array}
\right\}  \label{Hl}%
\end{equation}
as follows%
\[
\pi_{\lambda}\left(  \exp W\right)  \phi\left(
{\displaystyle\prod\limits_{k=1}^{d}}
\exp\left(  x_{k}X_{k}\right)  \right)  =\phi\left(  \exp\left(  -W\right)
{\displaystyle\prod\limits_{k=1}^{d}}
\exp\left(  x_{k}X_{k}\right)  \right)  .
\]
Next, we observe that the map $\beta:%
\mathbb{R}
^{d}\times\exp\left(  \mathfrak{z}\oplus\mathfrak{b}\right)  \rightarrow N$%
\[
\left(  \left(  x_{1},x_{2},\cdots,x_{d}\right)  ,\exp X\right)  \mapsto
\exp\left(  x_{1}X_{1}\right)  \exp\left(  x_{2}X_{2}\right)  \cdots
\exp\left(  x_{d}X_{d}\right)  \exp X
\]
is a diffeomorphism. Based on the properties of the Hilbert
space (\ref{Hl}), for $X\in\mathfrak{z}\oplus\mathfrak{b,}$ and $\phi
\in\mathbf{H}_{\lambda}$ if 
\[
n=\exp\left(  x_{1}X_{1}\right)  \exp\left(  x_{2}X_{2}\right)  \cdots
\exp\left(  x_{d}X_{d}\right)  \exp X
\]
then $\phi\left(  n\right)  =\phi\left(  \exp\left(  x_{1}X_{1}\right)
\exp\left(  x_{2}X_{2}\right)  \cdots\exp\left(  x_{d}X_{d}\right)  \right)
e^{-2\pi i\lambda\left(  X\right)  }.$ Thus, we may naturally identify the
Hilbert completion of $\mathbf{H}_{\lambda}$ with $$
L^{2}\left(
{\displaystyle\prod\limits_{k=1}^{d}}
\exp\left(
\mathbb{R}
X_{k}\right)  \right)  \cong L^{2}(\mathbb{R}^{d}).$$ Now, we will compute the action of $\pi_{\lambda}.$ Letting $Y_{j}%
\in\mathfrak{b},$ then $$
\pi_{\lambda}\left(  \exp\left(  l_{j}Y_{j}\right)  \right)  \phi\left(
{\displaystyle\prod\limits_{k=1}^{d}}
\exp\left(  x_{k}X_{k}\right)  \right)  =\phi\left(  \exp\left(  -l_{j}%
Y_{j}\right)
{\displaystyle\prod\limits_{k=1}^{d}}
\exp\left(  x_{k}X_{k}\right)  \right)  .$$ Next, let $\mathbf{x}=%
{\displaystyle\prod\limits_{k=1}^{d}}
\exp\left(  x_{k}X_{k}\right)  .$
\begin{align*}
\exp\left(  -l_{j}Y_{j}\right)  \mathbf{x}  &  =\mathbf{x}\left(
\mathbf{x}^{-1}\exp\left(  -l_{j}Y_{j}\right)  \mathbf{x}\right) \\
&  =\mathbf{x}\exp\left(  -l_{j}Y_{j}+\sum_{k=1}^{d}x_{k}l_{j}\left[
X_{k},Y_{j}\right]  \right) \\
&  =\mathbf{x}\exp\left(  -l_{j}Y_{j}\right)  \exp\left(  \sum_{k=1}^{d}%
x_{k}l_{j}\left[  X_{k},Y_{j}\right]  \right)  .
\end{align*}
Thus, $\pi_{\lambda}\left(  \exp l_{j}Y_{j}\right)  \phi\left(  x\right)
=e^{-2\pi i\left(  \sum_{k=1}^{d}x_{k}l_{j}\lambda\left[  X_{k},Y_{j}\right]
\right)  }\phi\left(  x\right)  $ and
\[
\pi_{\lambda}\left(  \exp\left(  -m_{1}X_{1}\right)  \right)  \phi\left(
(\exp\left(  x_{1}X_{1}\right)
{\displaystyle\prod\limits_{k=2}^{d}}
\exp\left(  x_{k}X_{k}\right)  \right)  =\phi\left(  \exp\left(  \left(
x_{1}-m_{1}\right)  X_{1}\right)
{\displaystyle\prod\limits_{k=2}^{d}}
\exp\left(  x_{k}X_{k}\right)  \right)  .
\]
Also, for $j>1,$ since%
\[
\exp\left(  -x_{r}X_{r}\right)  \exp\left(  -m_{j}X_{j}\right)  \exp\left(
x_{r}X_{r}\right)  =\exp\left(  -m_{j}X_{j}+x_{r}m_{j}\left[  X_{r}%
,X_{j}\right]  \right)
\]
then $\exp\left(  -m_{j}X_{j}\right)  \exp\left(  x_{r}X_{r}\right)
=\exp\left(  x_{r}X_{r}\right)  \exp\left(  -m_{j}X_{j}+x_{r}m_{j}\left[
X_{r},X_{j}\right]  \right)  $ and
\begin{align*}
\exp\left(  -m_{j}X_{j}\right)  \mathbf{x}  &  =\exp\left(  x_{1}X_{1}\right)
\exp\left(  x_{2}X_{2}\right)  \cdots\\
&  \exp\left(  \left(  x_{j}-m_{j}\right)  X_{j}\right)  \cdots\exp\left(
x_{d}X_{d}\right) \\
&  \times\exp\left(  \sum_{r=1}^{j-1}m_{j}x_{r}\left[  X_{r},X_{j}\right]
\right)  .
\end{align*}
Thus,
\begin{align*}
\pi_{\lambda}\left(  \exp\left(  -m_{j}X_{j}\right)  \right)  \phi\left(
\mathbf{x}\right)   &  =e^{-2\pi i\lambda\left(  \sum_{r=1}^{j-1}m_{j}%
x_{r}\left[  X_{r},X_{j}\right]  \right)  }\times\\
&  \phi(\exp\left(  x_{1}X_{1}\right)  \exp\left(  x_{2}X_{2}\right)
\cdots\exp\left(  \left(  x_{j}-m_{j}\right)  X_{j}\right) \\
&  \cdots\exp\left(  x_{d}X_{d}\right)  ).
\end{align*}
Finally, $\pi_{\lambda}\left(  \exp s_{j}Z_{j}\right)  \phi\left(  x\right)
=e^{2\pi i\lambda\left(  \exp s_{j}Z_{j}\right)  }\phi\left(  x\right)  .$ In
conclusion, identifying $%
\mathbb{R}
^{d}\ $with $N/\exp\left(  \mathfrak{z}\oplus\mathfrak{b}\right)  ,$
\[
\pi_{\lambda}\left(  \exp l_{j}Y_{j}\right)  \phi\left(  x\right)  =e^{-2\pi
i\lambda\left(  \sum_{k=1}^{d}x_{k}l_{j}\left[  X_{k},Y_{j}\right]  \right)
}\phi\left(  x\right)
\]
For $j=1,$ $\pi_{\lambda}\left(  \exp\left(  m_{1}X_{1}\right)  \right)
\phi\left(  x\right)  =\phi\left(  x_{1}-m_{1},x_{2},\cdots,x_{j},\cdots
x_{d}\right)  .$ For $j>1,$ we obtain%
\[
\pi_{\lambda}\left(  \exp\left(  m_{j}X_{j}\right)  \right)  \phi\left(
x\right)  =e^{-2\pi i\lambda\left(  \sum_{r=1}^{j-1}m_{j}x_{r}\left[
X_{r},X_{j}\right]  \right)  }\phi\left(  x_{1},x_{2},\cdots,x_{j}%
-m_{j},\cdots x_{d}\right)
\]
and finally, $
\pi_{\lambda}\left(  \exp s_{j}Z_{j}\right)  \phi\left(  x\right)  =e^{2\pi
i\lambda\left(  s_{j}Z_{j}\right)  }\phi\left(  x\right).$ Thus, the proposition is proved by putting the elements $\exp m_k X_k$ in the
appropriate order.
\end{proof}

\begin{example}
Let $N$ be a nilpotent Lie group with Lie algebra spanned by $Z,Y_{2}%
,Y_{1},X_{2},X_{1}$ with non-trivial Lie brackets $
\lbrack X_{1},X_{2}]=[X_{1},Y_{1}]=[X_{2},Y_{2}]=Z.$ The unitary dual of $N$ is parametrized by $$
\Sigma=\left\{  \lambda\in\mathfrak{n}^{\ast}:\lambda\left(  Z\right)
\neq0,\lambda\left(  Y_{2}\right)  =\lambda\left(  Y_{1}\right)
=\lambda\left(  X_{2}\right)  =\lambda\left(  X_{1}\right)  =0\right\}.$$ With some straightforward computations, we obtain that $$
\pi_{\lambda}(z_{2}Z_{2})\pi_{\lambda}(z_{1}Z_{1})\pi_{\lambda}(l_{2}Y_{2}%
)\pi_{\lambda}(l_{1}Y_{1})\pi_{\lambda}(k_{2}X_{2})\pi_{\lambda}(k_{1}%
X_{1})v(x_{1},x_{2})$$ is equal to
\[
e^{2\pi z_{2}i\lambda}e^{2\pi z_{1}i\lambda}e^{-2\pi ix_{2}l_{2}\lambda
}e^{-2\pi ix_{1}l_{1}\lambda}e^{-2\pi ix_{1}k_{2}\lambda}v(x_{1}-k_{1}%
,x_{2}-k_{2})
\]
where $v\in L^{2}(\mathbb{R}^{2}).$
\end{example}
\subsection{Proof of Theorem \ref{M1}}
Let $\Lambda$ be a full rank lattice in $\mathbb{R}^{2d}$ and $v\in
L^{2}\left(  \mathbb{R}^{d}\right)  $. We recall that the family of functions in
$L^{2}\left(  \mathbb{R}^{d}\right)  $: $
\mathcal{G}\left(  v,\Lambda\right)  =\left\{  e^{2\pi i\left\langle
k,x\right\rangle }v\left(  x-n\right)  :\left(  n,k\right)  \in\Lambda
\right\}$ is called a \textbf{Gabor system}. 
We are now ready to prove Theorem \ref{M1}.  We will show that if $\lambda\in\Sigma,$ and $v\in L^{2}\left(\mathbb{R}^{d}\right)  ,$ then $
\pi_{\lambda}\left(  \Gamma_{1}\right)  v=\mathcal{G}\left(v,  B\left(
\lambda\right)\mathbb{Z}^{2d}\right)$ where $B\left(  \lambda\right)  $ is a square matrix of order $2d$ described as follows.  \label{0}%
\begin{equation}
B\left(  \lambda\right)  =\left[
\begin{array}
[c]{cc}%
1_{d,d} & 0_{d,d}\\
-X\left(  \lambda\right)  & -S\left(  \lambda\right)
\end{array}
\right]  ,\text{ }S\left(  \lambda\right)  =\left[
\begin{array}
[c]{ccc}%
\lambda\left[  X_{1},Y_{1}\right]  & \cdots & \lambda\left[  X_{1}%
,Y_{d}\right] \\
\vdots & \ddots & \vdots\\
\lambda\left[  X_{d},Y_{1}\right]  & \cdots & \lambda\left[  X_{d}%
,Y_{d}\right]
\end{array}
\right]  , \label{MatrixB}%
\end{equation}
and $X\left(  \lambda\right)  $ is a matrix with entries in the dual of the
vector space $\left[  \mathfrak{a},\mathfrak{a}\right]  $ given by
\[
X\left(  \lambda\right)  =\left[
\begin{array}
[c]{ccccc}%
0 & \lambda\left[  X_{1},X_{2}\right]  & \lambda\left[  X_{1},X_{3}\right]  &
\cdots & \lambda\left[  X_{1},X_{d}\right] \\
\vdots & 0 & \lambda\left[  X_{2},X_{3}\right]  & \cdots & \lambda\left[
X_{2},X_{d}\right] \\
&  & \ddots & \cdots & \vdots\\
\vdots &  &  & 0 & \lambda\left[  X_{d-1},X_{d}\right] \\
0 & \cdots &  & \cdots & 0
\end{array}
\right] 
\]
\begin{proof}[Proof of Theorem \ref{M1}]
 Regarding $B\left(  \lambda\right)  $ as a linear operator acting on $\mathbb{R}^{2d}$ which we identify with
\[\mathbb{R}\text{-span }\left\{  X_{1},X_{2}\cdots X_{d},Y_{1},Y_{2},\cdots
,Y_{d}\right\}  ,
\]
we obtain
\[
B\left(  \lambda\right)  \left[
\begin{array}
[c]{c}%
m_{1}\\
\vdots\\
m_{d}\\
l_{1}\\
\vdots\\
l_{d}%
\end{array}
\right]  =\left[
\begin{array}
[c]{c}%
m_{1}\\
\vdots\\
m_{d}\\
-\sum_{k=1}^{d}l_{k}\lambda\left[  X_{1},Y_{k}\right]  -\sum_{k=2}^{d}%
m_{k}\lambda\left[  X_{1},X_{k}\right] \\
\vdots\\
-\sum_{k=1}^{d}l_{k}\lambda\left[  X_{d-1},Y_{k}\right]  -m_{k-1}%
\lambda\left[  X_{d-1},X_{d}\right] \\
\sum_{k=1}^{d}l_{k}\lambda\left[  X_{d},Y_{k}\right]
\end{array}
\right]  .
\]
Appealing to Proposition \ref{irreducible}, we compute
\begin{align*}
&  \pi_{\lambda}\left(  \exp\left(  l_{d}Y_{d}\right)  \cdots\exp\left(
l_{1}Y_{1}\right)  \exp\left(  m_{d}X_{d}\right)  \cdots\exp\left(  m_{1}%
X_{1}\right)  \right)  v\left(  x\right) \\
&  =e^{-2\pi i\left(  \sum_{j=1}^{d}\sum_{k=1}^{d}x_{k}l_{j}\lambda\left[
X_{k},Y_{j}\right]  +\sum_{j=2}^{d}\sum_{r=1}^{j-1}m_{j}x_{r}\left[
X_{j},X_{r}\right]  \right)  }\times v\left(  x_{1}-m_{1},\cdots,x_{d}%
-m_{d}\right)  .
\end{align*}
Factoring all the terms multiplying $x_{1},x_{2},\cdots,$ and $x_{d}$ in
\[
-\sum_{j=1}^{d}\sum_{k=1}^{d}x_{k}l_{j}\lambda\left[  X_{k},Y_{j}\right]
-\sum_{j=2}^{d}\sum_{r=1}^{j-1}m_{j}x_{r}\left[  X_{j},X_{r}\right]  ,
\]
we obtain that $\pi_{\lambda}\left(  \Gamma_{1}\right)  v=\mathcal{G}\left(v,
B\left(  \lambda\right)\mathbb{Z}^{2d}\right)  $ and
\[
\det B\left(  \lambda\right)  =-\det S\left(  \lambda\right)  \det\left(
1_{d,d}\right)  -\det X\left(  \lambda\right)  \times\det\left(
0_{d,d}\right)  =-\det S\left(  \lambda\right)  .
\]

\end{proof}

\begin{example}
Let $N$ be a simply connected, connected nilpotent Lie group with Lie algebra
spanned by the ordered basis
$
\{Z_{3},Z_{2},Z_{1},Y_{2},Y_{1},X_{2},X_{1}\}
$
with the following non-trivial Lie brackets
\begin{align*}
\left[  X_{1},X_{2}\right]   &  =Z_{3},\left[  X_{1},Y_{1}\right]
=Z_{1},\left[  X_{1},Y_{2}\right]  =Z_{2}, \left[  X_{2},Y_{1}\right]  =Z_{2},\left[  X_{2},Y_{2}\right]  =Z_{1}.
\end{align*}
Given $v\in L^{2}\left(\mathbb{R}^{4}\right)  ,$ we have $\pi_{\lambda}\left(  \Gamma_{1}\right)
v=\mathcal{G}\left(v,  B\left(  \lambda\right)\mathbb{Z}^{4}\right)  $ and
\[
B\left(  \lambda\right)  =\left[
\begin{array}
[c]{cccc}%
1 & 0 & 0 & 0\\
0 & 1 & 0 & 0\\
0 & -\lambda\left(  Z_{3}\right)  & -\lambda\left(  Z_{1}\right)  &
-\lambda\left(  Z_{2}\right) \\
0 & 0 & -\lambda\left(  Z_{2}\right)  & -\lambda\left(  Z_{1}\right)
\end{array}
\right]  .
\]

\end{example}
\subsection{Plancherel Theory}\label{Plancherel}
Now, we recall well-known facts about the Plancherel theory for the class of
groups considered in this paper. Assume that $N$ is endowed with its canonical
Haar measure which is the Lebesgue measure in our situation.  For
$\lambda=\left(  \lambda_{1},\cdots,\lambda_{n-2d},0,\cdots,0\right)
\in\Sigma,$ (see (\ref{Sigma}))
\begin{equation}
d\mu\left(  \lambda\right)  =\left\vert \det B\left(  \lambda\right)
\right\vert d\lambda=\left\vert \det S\left(  \lambda\right)  \right\vert
d\lambda\label{mu}%
\end{equation}
is the Plancherel measure (see chapter 4 in \cite{Corwin}), and the matrix
$B\left(  \lambda\right)  $ is defined in (\ref{MatrixB}). We have
\[
\mathcal{F}:L^{2}\left(  N\right)  \rightarrow\int_{\Sigma}^{\oplus}%
L^{2}\left(  \mathbb{R}^{d}\right)  \otimes L^{2}\left(  \mathbb{R}%
^{d}\right)  d\mu\left(  \lambda\right)
\]
where the Fourier transform is defined on $L^{2}(N)\cap L^{1}(N)$ by $$
\mathcal{F}\left(  f\right)  \left(  \lambda\right)  =\int_{\Sigma}  f\left(  n\right)
\pi_{\lambda}\left(  n\right)  dn$$ and the Plancherel transform $\mathcal{P}$ is the \textbf{extension} of the Fourier transform to
$L^{2}(N)$ inducing the equality
\[
\left\Vert f\right\Vert _{L^{2}\left(  N\right)  }^{2}=\int_{\Sigma}\left\Vert
\mathcal{P}\left(  f\right)  \left(  \lambda\right)  \right\Vert
_{\mathcal{HS}}^{2}d\mu\left(  \lambda\right)  .
\]
In fact, $||\cdot||_{\mathcal{HS}}$ denotes the Hilbert-Schmidt norm on
$L^{2}\left(  \mathbb{R}^{d}\right)  \otimes L^{2}\left(  \mathbb{R}%
^{d}\right)  $. We recall that the inner product of two rank-one operators in $L^{2}\left(
\mathbb{R}^{d}\right)  \otimes L^{2}\left(  \mathbb{R}^{d}\right)  $ is given
by $\left\langle u\otimes v,w\otimes y\right\rangle _{\mathcal{HS}%
}=\left\langle u,w\right\rangle _{L^{2}\left(  \mathbb{R}^{d}\right)
}\left\langle v,y\right\rangle _{L^{2}\left(  \mathbb{R}^{d}\right)  }.$ We
also have that $
L\cong\mathcal{P}L\mathcal{P}^{-1}=\int_{\Sigma}^{\oplus}\pi_{\lambda}%
\otimes\mathbf{1}_{L^{2}\left(  \mathbb{R}^{d}\right)  }d\mu\left(
\lambda\right),$ and $\mathbf{1}_{L^{2}\left(  \mathbb{R}^{d}\right)  }$ is the identity
operator on $L^{2}\left(  \mathbb{R}^{d}\right)  .$ Finally, for $\lambda
\in\Sigma,$ it is well-known that
\begin{equation}
\mathcal{P}(L(x)\phi)(\lambda)=\pi_{\lambda}(x)\circ\left(  \mathcal{P}%
\phi\right)  (\lambda). \label{rep}%
\end{equation}

\section{Reconstruction of Band-limited Vectors and Proof of Theorem \ref{M}}
\subsection{Properties of Band-limited Hilbert Subspaces}
We will start this section by introducing a natural concept of band-limitation on the
class of groups considered in this paper. The following set will be of special
interest, and we will be mainly interested in multiplicity-free subspaces.
Let
\[
\mathbf{E}=\left\{  \lambda\in\Sigma:\left\vert \det S\left(  \lambda\right)
\right\vert \leq1\right\}
\]
and let $\mathbf{m}$\textbf{ }be the Lebesgue measure on $\mathfrak{z}^{\ast
}.$ We remark that depending on the structure constants of the Lie algebra, the set
$\mathbf{E}$ is either bounded or unbounded. Furthermore, we need the
following lemma to hold.

\begin{lemma}
$\mathbf{E}$ is a set of positive Lebesgue measure
\end{lemma}

\begin{proof}
Since $\det S\left(  \lambda\right)  $ is a homogeneous polynomial, there
exists $a>0$ such that
\[
\left\{  \lambda\in\mathfrak{z}^{\ast}:\left\vert \lambda_{k}\right\vert
<a\text{ and }\left\vert \det S\left(  \lambda\right)  \right\vert
\neq0\right\}  \subset\mathbf{E.}%
\]
It is easy to see that
\[
\left\{  \lambda\in\mathfrak{z}^{\ast}:\left\vert \lambda_{k}\right\vert
<a\text{ and }\left\vert \det S\left(  \lambda\right)  \right\vert
\neq0\right\}  =\left\{  \lambda\in\mathfrak{z}^{\ast}:\left\vert \lambda
_{k}\right\vert <a\text{ }\right\}  -\left\{  \lambda\in\mathfrak{z}^{\ast
}:\det S\left(  \lambda\right)  =0\right\}  .
\]
However $\left\{  \lambda\in\mathfrak{z}^{\ast}:\det S\left(  \lambda\right)
=0\right\}  $ is a set of $\mathbf{m}$-measure zero. As a result,
\[
\mathbf{m}\left(  \left\{  \lambda\in\mathfrak{z}^{\ast}:\left\vert
\lambda_{k}\right\vert <a\text{ and }\left\vert \det S\left(  \lambda\right)
\right\vert \neq0\right\}  \right)  =\mathbf{m}\left(  \left\{  \lambda
\in\mathfrak{z}^{\ast}:\left\vert \lambda_{k}\right\vert <a\text{ }\right\}
\right)  .
\]
Thus, $\mathbf{m}\left(  \left\{  \lambda\in\mathfrak{z}^{\ast}:\left\vert
\lambda_{k}\right\vert <a\text{ and }\left\vert \det S\left(  \lambda\right)
\right\vert \neq0\right\}  \right)  >0$ and it follows that $\mathbf{E}$ is a
set of positive Lebesgue measure.
\end{proof}

\begin{definition}
Let $\mathbf{A}\subset\Sigma$ be a measurable bounded set$.$ We say a function
$f\in L^{2}(N)$ is $\mathbf{A}$-\textbf{band-limited} if its Plancherel
transform is supported on $\mathbf{A.\ }$Fix $\mathbf{u=}$ $\left\{
\mathbf{u}_{\lambda}:\lambda\in\Sigma\right\}  $ a measurable field of unit
vectors in $L^{2}\left(  \mathbb{R}^{d}\right)  $ which is parametrized by
$\Sigma.$ The Hilbert space
\[
\mathbf{H}_{\mathbf{u}}\text{ }\mathbf{=}\text{ }\mathcal{P}^{-1}\left(
\int_{\Sigma}^{\oplus}L^{2}\left(  \mathbb{R}^{d}\right)  \otimes
\mathbf{u}_{\lambda}\text{ }d\mu\left(  \lambda\right)  \right)
\]
is a multiplicity-free subspace of $L^{2}\left(  N\right)  .$ For any
measurable subset of $\mathbf{A}$ of $\Sigma,$ we define the Hilbert space%
\begin{equation}
\mathbf{H}_{\mathbf{u,A}}\text{ }\mathbf{=}\text{ }\mathcal{P}^{-1}\left(
\int_{\mathbf{A}}^{\oplus}L^{2}\left(  \mathbb{R}^{d}\right)  \otimes
\mathbf{u}_{\lambda}\text{ }d\mu\left(  \lambda\right)  \right)  . \label{H}%
\end{equation}

\end{definition}

Clearly $\mathbf{H}_{\mathbf{u,A}}$ is a Hilbert subspace of $L^{2}(N)$ which
contains vectors whose Fourier transforms are rank-one operators and are
supported on the set $\mathbf{A}$.  Next, we recall the following standard facts in frame theory. A sequence $\left\{  f_{n}:n\in\mathbb{Z}\right\}  $ of elements in a Hilbert
space $\mathbf{H}$ is called a \textbf{frame} \cite{Han,Han Yang Wang,Heil,Kornelson} if there are constant $A,B>0$
such that
\[
A\left\Vert f\right\Vert ^{2}\leq{\displaystyle\sum_n}\left\vert \left\langle f,f_{n}\right\rangle \right\vert ^{2}\leq B\left\Vert
f\right\Vert ^{2}\text{ for all }f\in\mathbf{H}.
\]
The numbers $A,B$ in the definition of a frame are called \textbf{lower} and
\textbf{upper bounds} respectively. A frame is a \textbf{tight frame} if $A=B$
and a \textbf{normalized tight frame} or \textbf{Parseval frame} if $A=B=1.$

\begin{definition}A set in a Hilbert space $\mathbf{H}$ is total if the closure of its linear span is equal to $\mathbf{H}.$\end{definition} 

We will need the following theorem known as the \textbf{Density Theorem for Lattices} (\cite{Heil} Theorem 10). 
\begin{theorem} (\textbf{Density Theorem} )\label{density} Let $v\in L^{2}\left(\mathbb{R}^{d}\right)  $ and let $\Lambda=A\mathbb{Z}^{2d}$ where $A$ is an invertible matrix of order $2d.$ Then the following holds

\begin{enumerate}
\item If $\left\vert \det\left(  A\right)  \right\vert >1$ then $\mathcal{G}%
\left(  v,\Lambda\right)  $ is not total in $L^{2}\left(\mathbb{R}^{d}\right)  .$

\item If $\mathcal{G}\left(  v,\Lambda\right)  $ is a frame for $L^{2}\left(\mathbb{R}^{d}\right)  $ then $0<\left\vert \det\left(  A\right)  \right\vert \leq1.$
\end{enumerate}
\end{theorem}

In light of the theorem above, we have the following. 

\begin{proposition}
Let $\mathbf{J}$ be a measurable subset of $\Sigma.$ If $\mathbf{J-E}$ is a
set of positive measure then it is not possible to find a function
$g\in\mathbf{H}_{\mathbf{u,J}}$ such that $L\left(  \Gamma\right)  g$ is total
in $\mathbf{H}_{\mathbf{u,J}}.$ In other words, the representation $(L,\mathbf{H}_{\mathbf{u,J}}) $ is not cyclic. 
\end{proposition}

\begin{proof}
To prove Part 1, if $\mathbf{J-E}$ is a non-null set, by the Density Theorem for lattices (see \ref{density}) $\pi_{\lambda}\left(
\Gamma\right)  \circ\mathcal{P}g\left(  \lambda\right)  =\pi_{\lambda}\left(
\Gamma\right)  u_{\lambda}\otimes\mathbf{u}_{\lambda}$ cannot be total in
$L^{2}\left(\mathbb{R}^{d}\right)  \otimes\mathbf{u}_{\lambda}$ for all $\lambda\in\mathbf{J-E}$
since $\pi_{\lambda}\left(  \Gamma\right)  u_{\lambda}=\mathcal{G}\left(
u_{\lambda}, B\left(  \lambda\right)  \mathbb{Z}%
^{2d}\right)  $ cannot be total in $L^{2}\left(\mathbb{R}^{d}\right)  $ for any $\lambda\in\mathbf{J-E.}$ Thus $\overline
{\mathcal{P}\left(  \text{\textrm{span} }\left(  L\left(  \Gamma\right)
g\right)  \right)  }$ is contained but not equal to $\int_{\mathbf{J}}%
^{\oplus}L^{2}\left(  \mathbb{R}^{d}\right)  \otimes\mathbf{u}_{\lambda}%
d\mu\left(  \lambda\right)  .$
\end{proof}

Let $\mathbf{K}$ be a measurable fundamental domain for $\mathbb{Z}^{n-2d}\cap\mathfrak{z}^{\ast}$ such that $\mathbf{m}\left(  \mathbf{K\cap
E}\right)  $ is positive. Clearly such set always exists. In fact, we define $\mathbf{K}$ to be the
unit cube around the zero linear functional in $\mathfrak{z}^{\ast}$ as follows:
\[
\mathbf{K=}\left\{  \lambda\in\mathfrak{z}^{\ast}:\lambda\left(  Z_{k}\right)
\in\left[  -\frac{1}{2},\frac{1}{2}\right]  \text{ for }1\leq k\leq
n-2d\right\}.
\] Put $
\mathbf{I}=\mathbf{E}\cap\mathbf{K}.$ 

\begin{definition}
We say a set $\mathcal{T}$ is a \textbf{tiling set} for a lattice
$\mathcal{L}$ if and only 
\begin{enumerate}
\item ${\displaystyle\bigcup\limits_{l\in\mathcal{L}}}\left(  \mathcal{T+}l\right)  =\mathbb{R}^{d}$ a.e. 

\item $\left(  \mathcal{T+}l\right)  \cap\left(  \mathcal{T+}l^{\prime
}\right)  $ has Lebesgue measure zero for any $l\neq l^{\prime}$ in
$\mathcal{L}.$
\end{enumerate}
\end{definition}

\begin{definition}
We say that $\mathcal{T}$ is a \textbf{packing set} for a lattice
$\mathcal{L}$ if and only if $\left(  \mathcal{T+}l\right)  \cap\left(
\mathcal{T+}l^{\prime}\right)  $ has Lebesgue measure zero for any $l\neq
l^{\prime}$ in $\mathcal{L}$. 
\end{definition}

Let $M$ be a matrix of order $d.$ We define the norm of $M$ as follows.
\[
\left\Vert M\right\Vert_{\infty} =\sup\left\{  Mx:x\in\mathbb{R}^{d},\left\Vert x\right\Vert _{\max}=1\right\}  \text{ where }\left\Vert
x\right\Vert _{\max}=\max_{1\leq k\leq d}\left\vert x_{k}\right\vert.
\]
Now, put
\[
\mathbf{Q=}\left\{  \lambda\in\Sigma\mathbf{:}\text{ }\left\Vert S\left(
\lambda\right)  ^{Tr}\right\Vert_{\infty} <1\right\}  .
\]
It is clear that $\mathbf{Q}$ is a set of positive measure. 

\begin{lemma}\label{gab}
For any $\lambda\in\mathbf{Q,}$ then $\left[  -\frac{1}{2},\frac{1}{2}\right]
^{d}$ is a packing set for $S\left(  \lambda\right)  ^{-Tr}\mathbb{Z}^{d}$ and a tiling set for $\mathbb{Z}^{d}.$
\end{lemma}

\begin{proof}
Clearly $\left[  -\frac{1}{2},\frac{1}{2}\right)  ^{d}$ is a tiling set for $\mathbb{Z}^{d}.$ To show that the lemma holds, it suffices to show that $\left[
-\frac{1}{2},\frac{1}{2}\right)  ^{d}$ is a packing set for $S\left(
\lambda\right)  ^{-Tr}\mathbb{Z}^{d}.$ Let us suppose that there exist $\kappa_{1},\kappa_{2}\in S\left(
\lambda\right)  ^{-Tr}\mathbb{Z}^{d}$ and $\sigma_{1},\sigma_{2}\in\left[  -\frac{1}{2},\frac{1}{2}\right]
^{d},$ $\sigma_{1}\neq\sigma_{2}$ such that $\sigma_{1}+\kappa_{1}=\sigma
_{2}+\kappa_{2}.$ Then there exist $j_{2},j_{1}\in\mathbb{Z}^{d}$ such that $\sigma_{1}-\sigma_{2}=\left(  S\left(  \lambda\right)
^{Tr}\right)  ^{-1}\left(  j_{2}-j_{1}\right)  .$ So $S\left(  \lambda\right)
^{Tr}\left(  \sigma_{1}-\sigma_{2}\right)  =j_{2}-j_{1}$ and $\left\Vert
S\left(  \lambda\right)  ^{Tr}\left(  \sigma_{1}-\sigma_{2}\right)
\right\Vert _{\max}=\left\Vert j_{2}-j_{1}\right\Vert _{\max}.$ Since
$j_{2}-j_{1}\neq0$ then $$\left\Vert S\left(  \lambda\right)  ^{Tr}\left(
\sigma_{1}-\sigma_{2}\right)  \right\Vert _{\max}\geq1$$ and $\left\Vert
S\left(  \lambda\right)  ^{Tr}\left(  \sigma_{1}-\sigma_{2}\right)
\right\Vert _{\max}\leq\left\Vert S\left(  \lambda\right)  ^{Tr}\right\Vert_{\infty}
<1.$ Thus $$1>\left\Vert S\left(  \lambda\right)  ^{Tr}\left(  \sigma
_{1}-\sigma_{2}\right)  \right\Vert _{\max}\geq1.$$ That would be a
contradiction. 
\end{proof}

From now on, we will assume that $\mathbf{I}$ is replaced with $\mathbf{Q}\cap\mathbf{I}.$  \begin{example}
Let $N$ be a nilpotent Lie group with Lie algebra $\mathfrak{n}$ spanned by
\[
\left\{  Z_{2},Z_{1},Y_{2},Y_{1},X_{2},X_{1}\right\}
\]
with the following non-trivial Lie brackets.
\begin{align*}
\left[  X_{1},X_{2}\right]    & =Z_{2},\left[  X_{1},Y_{1}\right]  =Z_{1},
\left[  X_{2},Y_{1}\right]     =Z_{2},\left[  X_{3},Y_{1}\right]  =-Z_{1},
\left[  X_{1},Y_{2}\right]    & =Z_{2},\left[  X_{2},Y_{2}\right]  =-Z_{1}\\
\left[  X_{3},Y_{2}\right]     &=Z_{1},\left[  X_{1},Y_{3}\right]  =-Z_{1},
\left[  X_{2},Y_{3}\right]     =Z_{1},\left[  X_{3},Y_{3}\right]  =Z_{2}.
\end{align*}
Let $\lambda\in\mathfrak{n}^{\ast}$, we write $\lambda=\left(  \lambda
_{1},\lambda_{2},\lambda_{3},\cdots,\lambda_{n}\right)  $ where $\lambda
_{k}=\lambda\left(  Z_{k}\right)  .$ Then
\[
\mathbf{I}=\left\{
\begin{array}
[c]{c}%
\left(  \lambda_{1},\lambda_{2},0,\cdots,0\right)  \in\mathfrak{z}^{\ast
}:-3\lambda_{1}^{2}\lambda_{2}-\lambda_{2}^{3}\neq0, \left\vert 3\lambda_{1}^{2}\lambda_{2}+\lambda_{2}^{3}\right\vert
\leq1,2\left\vert \lambda_{1}\right\vert +\left\vert \lambda_{2}\right\vert <1\\ -1/2\leq \lambda_1, \lambda_2\leq 1/2
\end{array}
\right\} .
\]
\end{example}
Next, we define the unitary operator $\mathcal{U}:L^{2}\left(\mathbb{R}^{d}\right)  \xrightarrow{\hspace*{1cm}} L^{2}\left(\mathbb{R}^{d}\right)  $ such that
\[
\mathcal{U}f\left(  t\right)  =e^{-2\pi i\left\langle t,X\left(
\lambda\right)  t\right\rangle }f\left(  t\right)
\]

\begin{lemma} \label{pf}
For every linear functional $\lambda\in\mathbf{I}$,$$
\mathcal{G}\left(  \left\vert \det S\left(  \lambda\right)  \right\vert
^{1/2}\mathcal{U}\chi_{\left[  -1/2,1/2\right]  ^{d}},B\left(  \lambda\right)
\mathbb{Z}^{2d}\right)$$ is a Parseval frame in $L^{2}\left(\mathbb{R}^{d}\right)  .$
\end{lemma}

\begin{proof}
Given $v\in L^{2}\left(\mathbb{R}^{d}\right),$ we write $M_{l}v\left(  t\right)  =e^{2\pi i\left\langle
k,t\right\rangle }v\left(  t\right)  $ and $T_{k}v\left(  t\right)  =v\left(
t-k\right)  .$ Thus,
\[
\mathcal{G}\left(  v,B\left(  \lambda\right)\mathbb{Z}^{2d}\right)  =\left\{  M_{-S\left(  \lambda\right)  l-X\left(  \lambda
\right)  k}T_{k}v:\left(  k,l\right)  \in\mathbb{Z}^{2d}\right\}  .
\]
Next, we will see that $\mathcal{U}M_{-S\left(  \lambda\right)  l}%
T_{k}\mathcal{U}^{-1}v\left(  t\right)  =M_{-S\left(  \lambda\right)
l-X\left(  \lambda\right)  k}T_{k}v\left(  t\right)  .$ Indeed,
\begin{align*}
\mathcal{U}M_{-S\left(  \lambda\right)  l}T_{k}\mathcal{U}^{-1}v\left(
t\right)    & =e^{-2\pi i\left\langle t,X\left(  \lambda\right)
t\right\rangle }M_{-S\left(  \lambda\right)  l}T_{k}\mathcal{U}^{-1}v\left(
t\right)  \\
& =e^{-2\pi i\left\langle t,X\left(  \lambda\right)  t\right\rangle }e^{-2\pi
i\left\langle S\left(  \lambda\right)  l,t\right\rangle }\mathcal{U}%
^{-1}v\left(  t-k\right)  \\
& =e^{-2\pi i\left\langle t,X\left(  \lambda\right)  t\right\rangle }e^{-2\pi
i\left\langle S\left(  \lambda\right)  l,t\right\rangle }e^{2\pi i\left\langle
t,X\left(  \lambda\right)  \left(  t-k\right)  \right\rangle }v\left(
t-k\right)  \\
& =e^{-2\pi i\left\langle t,X\left(  \lambda\right)  t\right\rangle }e^{-2\pi
i\left\langle S\left(  \lambda\right)  l,t\right\rangle }e^{2\pi i\left\langle
t,X\left(  \lambda\right)  t\right\rangle }e^{-2\pi i\left\langle t,X\left(
\lambda\right)  k\right\rangle }v\left(  t-k\right)  \\
& =e^{-2\pi i\left\langle S\left(  \lambda\right)  l,t\right\rangle }e^{-2\pi
i\left\langle t,X\left(  \lambda\right)  k\right\rangle }v\left(  t-k\right)
\\
& =M_{-S\left(  \lambda\right)  l-X\left(  \lambda\right)  k}T_{k}v\left(
t\right)  .
\end{align*}
Put $w=\mathcal{U}^{-1}v.$ Then
$
\mathcal{U}M_{-S\left(  \lambda\right)  l}T_{k}\mathcal{U}^{-1}v=\mathcal{U}%
M_{-S\left(  \lambda\right)  l}T_{k}w.
$
Now, let $w=\left\vert \det S\left(  \lambda\right)  \right\vert ^{1/2}%
\chi_{E\left(  \lambda\right)  }$ such that $E\left(  \lambda\right)  $ is a
tiling set for $\mathbb{Z}^{d}$ and a packing set for $\left(  S\left(  \lambda\right)  ^{Tr}\right)
^{-1}\mathbb{Z}^{d}.$ Then
$
\mathcal{G}\left(  w,\mathbb{Z}^{d}\times S\left(  \lambda\right)\mathbb{Z}^{d}\right)
$
is a Parseval frame in $L^{2}\left(
\mathbb{R}
^{d}\right)  $  (see Proposition 3.1 \cite{ Pfander}) and  $
\mathcal{G}\left(  \left\vert \det S\left(  \lambda\right)  \right\vert
^{1/2}\mathcal{U}\chi_{E\left(  \lambda\right)  },B\left(  \lambda\right)
\mathbb{Z}
^{2d}\right)$ is a Parseval frame in $L^{2}\left(
\mathbb{R}
^{d}\right).  $ The proof of the lemma is completed by replacing $E\left(
\lambda\right)  $ with $\left[  -1/2,1/2\right]  ^{d}.$ 
\end{proof}

\begin{proposition}\label{pro}
Let $\mathbf{H}_{\mathbf{u,I}}$ be a Hilbert space consisting of $\mathbf{I}%
$-band-limited functions. There exists a function $f$ in $\mathbf{H}%
_{\mathbf{u,I}}$ such that $L\left(  \Gamma\right)  f$ is a Parseval frame in
$\mathbf{H}_{\mathbf{u,I}}$ and
$\left\Vert f\right\Vert _{\mathbf{H}%
_{\mathbf{u},\mathbf{I}}}^{2}=\mu\left(  \mathbf{I}\right)  .$

\end{proposition}

\begin{proof}
Define $f\in\mathbf{H}_{\mathbf{u,I}}$ such that $
\mathcal{P}f\left(  \lambda\right)  =\left\vert \det B\left(  \lambda\right)
\right\vert ^{-1/2}\phi\left(  \lambda\right)  \otimes\mathbf{u}_{\lambda}$ such that $$ \phi(\lambda)=\left\vert \det S\left(  \lambda\right)  \right\vert
^{1/2}\mathcal{U}\chi_{\left[  -1/2,1/2\right]  ^{d}}. $$ We recall that $
\mathcal{P}\left(  L\left(  \gamma\right)  f\right)  \left(  \lambda\right)
=\pi_{\lambda}\left(  \gamma\right)  \circ\mathcal{P}f\left(  \lambda\right)
=\left\vert \det B\left(  \lambda\right)  \right\vert ^{-1/2}\pi_{\lambda
}\left(  \gamma\right)  \phi\left(  \lambda\right)  \otimes\mathbf{u}%
_{\lambda}.$ Let $g$ be any function in $\mathbf{H}_{\mathbf{u,I}}$ such that
$\mathcal{P}g\left(  \lambda\right)  =u_{\lambda}\otimes\mathbf{u}_{\lambda}.$
Next, we write $\gamma\in\Gamma$ such that $\gamma=k\eta,$ where $k$ is in the
center of the Lie group $N$ and $\eta$ is in $
\Gamma_{1}=\exp\mathbb{Z}Y_{d}\cdots\exp\mathbb{Z}Y_{1}\exp\mathbb{Z}%
X_{d}\cdots\exp\mathbb{Z}X_{1}.$
\begin{align}
\sum_{\gamma\in\Gamma}\left\vert \left\langle g,L\left(  \gamma\right)
f\right\rangle _{\mathbf{H}_{\mathbf{u,I}}}\right\vert ^{2}  &  =\sum
_{\gamma\in\Gamma}\left\vert \int_{\mathbf{I}}\left\langle u_{\lambda}%
\otimes\mathbf{u}_{\lambda},\left\vert \det B\left(  \lambda\right)
\right\vert \pi_{\lambda}\left(  \gamma\right)  \left(  \left\vert \det
B\left(  \lambda\right)  \right\vert ^{-1/2}\phi\left(  \lambda\right)
\right)  \otimes\mathbf{u}_{\lambda}\right\rangle _{\mathcal{HS}}%
d\lambda\right\vert ^{2}\nonumber\\
&  =\sum_{\gamma\in\Gamma}\left\vert \int_{\mathbf{I}}\left\langle u_{\lambda
}\otimes\mathbf{u}_{\lambda},\pi_{\lambda}\left(  \gamma\right)  \left(
\left\vert \det B\left(  \lambda\right)  \right\vert ^{1/2}\phi\left(
\lambda\right)  \right)  \otimes\mathbf{u}_{\lambda}\right\rangle
_{\mathcal{HS}}d\lambda\right\vert ^{2}\nonumber\\
&  =\sum_{\gamma\in\Gamma}\left\vert \int_{\mathbf{I}}\left\langle u_{\lambda
}\otimes\mathbf{u}_{\lambda},\pi_{\lambda}\left(  \gamma\right)  \left(
\left\vert \det B\left(  \lambda\right)  \right\vert ^{1/2}\phi\left(
\lambda\right)  \right)  \otimes\mathbf{u}_{\lambda}\right\rangle
_{\mathcal{HS}}d\lambda\right\vert ^{2}\nonumber\\
&  =\sum_{\gamma\in\Gamma}\left\vert \int_{\mathbf{I}}\left\langle u_{\lambda
},\pi_{\lambda}\left(  \gamma\right)  \left\vert \det B\left(  \lambda\right)
\right\vert ^{1/2}\phi\left(  \lambda\right)  \right\rangle _{L^{2}\left(
\mathbb{R}^{d}\right)  }\left\langle \mathbf{u}_{\lambda},\mathbf{u}_{\lambda
}\right\rangle _{L^{2}\left(  \mathbb{R}^{d}\right)  }d\lambda\right\vert
^{2}\nonumber\\
&  =\sum_{\gamma\in\Gamma}\left\vert \int_{\mathbf{I}}\left\langle u_{\lambda
},\left\vert \det B\left(  \lambda\right)  \right\vert ^{1/2}\pi_{\lambda
}\left(  \gamma\right)  \phi\left(  \lambda\right)  \right\rangle
_{L^{2}\left(  \mathbb{R}^{d}\right)  }d\lambda\right\vert ^{2}\nonumber\\
&  =\sum_{\eta\in\Gamma_{1}}\sum_{k\in\mathbb{Z}^{n-2d}}\left\vert
\int_{\mathbf{I}}e^{-2\pi i\left\langle k,\lambda\right\rangle }\left\langle
u_{\lambda},\left\vert \det B\left(  \lambda\right)  \right\vert ^{1/2}%
\pi_{\lambda}\left(  \eta\right)  \phi\left(  \lambda\right)  \right\rangle
_{L^{2}\left(  \mathbb{R}^{d}\right)  }d\lambda\right\vert ^{2}. \label{down}%
\end{align}
Since $\left\{  e^{-2\pi i\left\langle k,\lambda\right\rangle }\chi
_{\mathbf{I}}\left(  \lambda\right)  :k\in%
\mathbb{Z}
\right\}  $ is a Parseval frame in $L^{2}\left(  \mathbf{I}\right)  ,$
letting
\[
c_{\eta}\left(  \lambda\right)  =\left\langle u_{\lambda},\left\vert \det
B\left(  \lambda\right)  \right\vert ^{1/2}\pi_{\lambda}\left(  \eta\right)
\phi\left(  \lambda\right)  \right\rangle _{L^{2}\left(
\mathbb{R}
^{d}\right)  },
\]
Equation (\ref{down}) becomes
\[
\sum_{\eta\in\Gamma_{1}}\sum_{k\in%
\mathbb{Z}
^{n-2d}}\left\vert \int_{\mathbf{I}}e^{2\pi i\left\langle k,\lambda
\right\rangle }c_{\eta}\left(  \lambda\right)  d\lambda\right\vert ^{2}%
=\sum_{\eta\in\Gamma_{1}}\sum_{k\in%
\mathbb{Z}
^{n-2d}}\left\vert \widehat{c}_{\eta}\left(  k\right)  \right\vert ^{2}%
=\sum_{\eta\in\Gamma_{1}}\left\Vert c_{\eta}\right\Vert _{L^{2}\left(
\mathbf{I}\right)  }^{2}.
\]
Next,
\begin{align*}
\sum_{\gamma\in\Gamma}\left\vert \left\langle g,L\left(  \gamma\right)
f\right\rangle _{\mathbf{H}_{\mathbf{u,I}}}\right\vert ^{2}  &  =\sum_{\eta
\in\Gamma_{1}}\int_{\mathbf{I}}\left\vert \left\langle u_{\lambda},\left\vert
\det B\left(  \lambda\right)  \right\vert ^{1/2}\pi_{\lambda}\left(
\eta\right)  \phi\left(  \lambda\right)  \right\rangle _{L^{2}\left(
\mathbb{R}
^{d}\right)  }\right\vert ^{2}d\lambda\\
&  =\int_{\mathbf{I}}\sum_{\eta\in\Gamma_{1}}\left\vert \left\langle
u_{\lambda},\left\vert \det B\left(  \lambda\right)  \right\vert ^{1/2}%
\pi_{\lambda}\left(  \eta\right)  \phi\left(  \lambda\right)  \right\rangle
_{L^{2}\left(
\mathbb{R}
^{d}\right)  }\right\vert ^{2}d\lambda\\
&  =\int_{\mathbf{I}}\sum_{\eta\in\Gamma_{1}}\left\vert \left\langle
u_{\lambda},\pi_{\lambda}\left(  \eta\right)  \phi\left(  \lambda\right)
\right\rangle _{L^{2}\left(
\mathbb{R}
^{d}\right)  }\right\vert ^{2}\left\vert \det B\left(  \lambda\right)
\right\vert d\lambda.
\end{align*}
Using the fact that $\mathcal{G}\left(  \phi\left(  \lambda\right)  ,B\left(
\lambda\right)
\mathbb{Z}
^{2d}\right)  $ is a Parseval frame for almost every $\lambda\in\mathbf{I}$ (see Lemma \ref{pf}), we obtain
\begin{align*}
\sum_{\eta\in\Gamma_{1}}\left\vert \left\langle u_{\lambda},\pi_{\lambda
}\left(  \eta\right)  \phi\left(  \lambda\right)  \right\rangle _{L^{2}\left(
%
\mathbb{R}
^{d}\right)  }\right\vert ^{2}  &  =\left\Vert u_{\lambda}\right\Vert
_{L^{2}\left(
\mathbb{R}
^{d}\right)  }^{2}\\
&  =\left\Vert u_{\lambda}\right\Vert _{L^{2}\left(
\mathbb{R}
^{d}\right)  }^{2}\left\Vert \mathbf{u}_{\lambda}\right\Vert _{L^{2}\left(
\mathbb{R}
^{d}\right)  }^{2}\\
&  =\left\Vert u_{\lambda}\otimes\mathbf{u}_{\lambda}\right\Vert
_{\mathcal{HS}}^{2}%
\end{align*}
and
\[
\sum_{\gamma\in\Gamma}\left\vert \left\langle g,L\left(  \gamma\right)
f\right\rangle _{\mathbf{H}_{\mathbf{u,I}}}\right\vert ^{2}=\int_{\mathbf{I}%
}\left\Vert \mathcal{P}g\left(  \lambda\right)  \right\Vert _{\mathcal{HS}%
}^{2}\left\vert \det B\left(  \lambda\right)  \right\vert d\lambda=\left\Vert
g\right\Vert _{\mathbf{H}_{\mathbf{u,I}}}^{2}.
\]
Now, to make sure that $f\in\mathbf{H}_{\mathbf{u,I}},$ we will show that its
norm is finite. Clearly,
\begin{align*}
\left\Vert f\right\Vert _{\mathbf{H}_{\mathbf{u,I}}}^{2}  &  =\int
_{\mathbf{I}}\left\Vert \mathcal{P}f\left(  \lambda\right)  \right\Vert
_{\mathcal{HS}}^{2}\left\vert \det B\left(  \lambda\right)  \right\vert
d\lambda\\
&  =\int_{\mathbf{I}}\left\Vert \left\vert \det B\left(  \lambda\right)
\right\vert ^{-1/2}\phi\left(  \lambda\right)  \otimes\mathbf{u}_{\lambda
}\right\Vert _{\mathcal{HS}}^{2}\left\vert \det B\left(  \lambda\right)
\right\vert d\lambda\\
&  =\int_{\mathbf{I}}\left\Vert \phi\left(  \lambda\right)  \otimes
\mathbf{u}_{\lambda}\right\Vert _{\mathcal{HS}}^{2}\left\vert \det B\left(
\lambda\right)  \right\vert ^{-1}\left\vert \det B\left(  \lambda\right)
\right\vert d\lambda\\
&  =\int_{\mathbf{I}}\left\Vert \phi\left(  \lambda\right)  \otimes
\mathbf{u}_{\lambda}\right\Vert _{\mathcal{HS}}^{2}d\lambda\\
&  =\int_{\mathbf{I}}\left\Vert \phi\left(  \lambda\right)  \right\Vert
_{L^{2}\left(
\mathbb{R}
^{d}\right)  }^{2}d\lambda.
\end{align*}
Since $\mathbf{I}$ is bounded then $\left\Vert f\right\Vert _{\mathbf{H}%
_{\mathbf{u,I}}}^{2}$ is clearly finite. In fact $\left\Vert \phi\left(  \lambda\right)  \right\Vert ^{2}=\left\vert
\det S\left(  \lambda\right)  \right\vert =\left\vert \det B\left(
\lambda\right)  \right\vert $ and $\left\Vert f\right\Vert _{\mathbf{H}%
_{\mathbf{u},\mathbf{I}}}^{2}=\mu\left(  \mathbf{I}\right)  .$
\end{proof}

\subsection{Proof of Theorem \ref{M}}
Finally, we are able to offer a proof of Theorem \ref{M}.
\begin{proof}[Proof of Theorem \ref{M}] Let $f$ be a function in $\mathbf{H}_{\mathbf{u},\mathbf{I}}$ such that  $
\mathcal{P}f\left(  \lambda\right)  =\mathcal{U}\chi_{\left[  -1/2,1/2\right]
^{d}}\otimes\mathbf{u}_{\lambda}.$ We recall  the coefficient function $V_{f}:\mathbf{H}_{\mathbf{u},\mathbf{I}%
}\xrightarrow{\hspace*{1cm}} L^{2}\left(  N\right)  $ such that $V_{f}h\left(  x\right)
=\left\langle h,L\left(  x\right)  f\right\rangle=h\ast f^{\ast}$ where $\ast$ is the convolution operation and $f^{\ast}(x)=\overline{f(x^{-1})}.$ We will first show that  $V_{f}$ is an isometry. In other words, $f$ is an admissible vector. 
\begin{align*}
\left\Vert V_{f}h\right\Vert ^{2}  & =\int_{\mathbf{I}}\left\Vert
\mathcal{P}h\left(  \lambda\right)  \circ\mathcal{P}\left(  f^{\ast}\right)
\left(  \lambda\right)  \right\Vert _{\mathcal{HS}}^{2}\text{ }d\mu\left(
\lambda\right)  \\
& =\int_{\mathbf{I}}\left\Vert \mathcal{P}h\left(  \lambda\right)
\circ\mathcal{P}\left(  f^{\ast}\right)  \left(  \lambda\right)  \right\Vert
_{\mathcal{HS}}^{2}\text{ }d\mu\left(  \lambda\right)  \\
& =\int_{\mathbf{I}}\left\Vert \mathcal{P}h\left(  \lambda\right)  \right\Vert
_{\mathcal{HS}}^{2}\text{ }\left\Vert \mathcal{P}\left(  f^{\ast}\right)
\left(  \lambda\right)  \right\Vert _{\mathcal{HS}}^{2}d\mu\left(
\lambda\right)  \\
& =\int_{\mathbf{I}}\left\Vert \mathcal{P}h\left(  \lambda\right)  \right\Vert
_{\mathcal{HS}}^{2}\text{ }d\mu\left(  \lambda\right)  \\
& =\left\Vert h\right\Vert ^{2}.
\end{align*}
The third equality above is justified because, the operators involved are rank-one operators. Since $f$ is an admissible vector for the representation $\left(
L,\mathbf{H}_{\mathbf{u},\mathbf{I}}\right)  $ and since $L\left(
\Gamma\right)  f$ is a Parseval frame then according to Proposition 2.54. \cite{Fuhr cont}, $V_f(\mathbf{H}_{\mathbf{u},\mathbf{I}})$ is a sampling space with respect to $\Gamma$ with sinc function $V_f (f).$
\end{proof}

\begin{example}
Let $N$ a be Lie group with Lie algebra $\mathfrak{n}$ spanned by $$\left\{  Z_{3},Z_{2},Z_{1},Y_{2},Y_{1},X_{2},X_{1}\right\}  $$ with the
following non-trivial Lie brackets.
\begin{align*}
\left[  X_{1},Y_{1}\right]    & =Z_{1},\left[  X_{1},Y_{2}\right]
=Z_{2},\left[  X_{2},Y_{1}\right]  =Z_{2}, \left[  X_{2},Y_{2}\right]  =Z_{3},\left[  X_{1},X_{2}\right]
=Z_{1}-Z_{3}.
\end{align*}
Define a bounded subset $\mathbf{I}$ of $\mathfrak{z}^{\ast}$ given by
\[
\mathbf{I=}\left\{
\begin{array}
[c]{c}%
\lambda\in\mathfrak{z}^{\ast}:\lambda\left(  Z_{1}\right)  \lambda\left(
Z_{3}\right)  -\lambda\left(  Z_{2}\right)  ^{2}\neq0,\left\vert
\lambda\left(  Z_{1}\right)  \lambda\left(  Z_{3}\right)  -\lambda\left(
Z_{2}\right)  ^{2}\right\vert \leq1\\
\max\left\{  \left\vert \lambda\left(  Z_{1}\right)  +\lambda\left(
Z_{2}\right)  \right\vert ,\left\vert \lambda\left(  Z_{2}\right)
+\lambda\left(  Z_{3}\right)  \right\vert \right\}  <1\\
\left(  \lambda\left(  Z_{1}\right)  ,\lambda\left(  Z_{2}\right)
,\lambda\left(  Z_{3}\right)  \right)  \in\left[  -1/2,1/2\right]  ^{3}%
\end{array}
\right\}  .
\]
Put $f\in$ $\mathbf{H}_{\mathbf{u},\mathbf{I}}$ such that $
\widehat{f}\left(  \lambda\right)  =e^{-2\pi i\left(  \lambda\left(
Z_{1}\right)  t_{1}t_{2}-\lambda\left(  Z_{3}\right)  t_{1}t_{2}\right)  }%
\chi_{\left[  -1/2,1/2\right]  ^{2}}\left(  t_{1},t_{2}\right)  \otimes
\mathbf{u}_{\lambda}$ where $\left\{  \mathbf{u}_{\lambda}:\lambda\in\mathbf{I}\right\}  \ $is a family of
unit vectors in $L^{2}\left(
\mathbb{R}
^{2}\right)  .$ Then $V_{f}\left(  \mathbf{H}_{\mathbf{u},\mathbf{I}}\right)  $ is a
sampling space with respect to the discrete set
\begin{align*}
\Gamma & =\exp\left(
\mathbb{Z}
Z_{3}\right)  \exp\left(
\mathbb{Z}
Z_{2}\right)  \exp\left(
\mathbb{Z}
Z_{1}\right)  \exp\left(
\mathbb{Z}
Y_{2}\right)  \\
& \exp\left(
\mathbb{Z}
Y_{1}\right)  \exp\left(
\mathbb{Z}
X_{2}\right)  \exp\left(
\mathbb{Z}
X_{1}\right)
\end{align*}
with sinc function $s=V_{f}f.$ Thus given any $h\in V_{f}\left(
\mathbf{H}_{\mathbf{u},\mathbf{I}}\right)  ,$ $h$ is determined by its sampled values
$\left(  h\left(  \gamma\right)  \right)  _{\gamma\in\Gamma}$ and $
h\left(  x\right)  =\sum_{\gamma\in\Gamma}h\left(  \gamma\right)  s\left(
\gamma^{-1}x\right).$

\end{example}
\section*{Acknowledgments}
The author thanks the anonymous reviewer for a careful and thorough reading. His suggestions and corrections greatly improved the quality of the paper.

\end{document}